\documentclass[12pt,american]{amsart}
\usepackage[T1]{fontenc}
\usepackage[latin9]{inputenc}
\usepackage{geometry}
\usepackage{xcolor}
\usepackage{pdfcolmk}
\definecolor{note_fontcolor}{rgb}{0.800781, 0.800781, 0.800781}
\usepackage{array}
\usepackage{calc}
\usepackage{multirow}
\usepackage{amstext}
\usepackage{amsthm}
\usepackage{amssymb}
\usepackage{stmaryrd}
\usepackage{graphicx}
\PassOptionsToPackage{normalem}{ulem}
\usepackage{ulem}

\makeatletter

\providecommand{\tabularnewline}{\\}
\newenvironment{lyxgreyedout}
  {\textcolor{note_fontcolor}\bgroup\ignorespaces}
  {\ignorespacesafterend\egroup}
\providecolor{lyxadded}{rgb}{0,0,1}
\providecolor{lyxdeleted}{rgb}{1,0,0}

\DeclareRobustCommand{\lyxsout}[1]{\ifx\\#1\else\sout{#1}\fi}

\numberwithin{equation}{section}
\numberwithin{figure}{section}
\numberwithin{table}{section}
\theoremstyle{plain}
\newtheorem{thm}{\protect\theoremname}[section]
\theoremstyle{definition}
\newtheorem{example}[thm]{\protect\examplename}
\theoremstyle{plain}
\newtheorem{prop}[thm]{\protect\propositionname}
\theoremstyle{definition}
\newtheorem{defn}[thm]{\protect\definitionname}
\theoremstyle{plain}
\newtheorem{conjecture}[thm]{\protect\conjecturename}
\theoremstyle{remark}
\newtheorem{rem}[thm]{\protect\remarkname}
\theoremstyle{plain}
\newtheorem{lem}[thm]{\protect\lemmaname}
\theoremstyle{plain}
\newtheorem{cor}[thm]{\protect\corollaryname}
\theoremstyle{remark}
\newtheorem{note}[thm]{\protect\notename}
\theoremstyle{plain}
\newtheorem{question}[thm]{\protect\questionname}

\usepackage{graphicx}
\usepackage{mathtools}
\usepackage{tikz}

\newtheoremstyle{plain}
  {5pt plus 1pt minus 1pt}   
  {5pt plus 1pt minus 1pt}   
  {}  
  {0pt}       
  {\bfseries} 
  {.}         
  {5pt plus 1pt minus 1pt} 
  {}          

\usepackage{enumitem}
\setlist{itemsep=0pt,topsep=0pt,parsep=1pt,partopsep=0pt}

\renewenvironment{lyxgreyedout}{\textcolor{olive}\bgroup}{\egroup}
\makeatother

\usepackage{babel}
\usepackage{listings}
\providecommand{\conjecturename}{Conjecture}
\providecommand{\corollaryname}{Corollary}
\providecommand{\definitionname}{Definition}
\providecommand{\examplename}{Example}
\providecommand{\lemmaname}{Lemma}
\providecommand{\notename}{Note}
\providecommand{\propositionname}{Proposition}
\providecommand{\questionname}{Question}
\providecommand{\remarkname}{Remark}
\providecommand{\theoremname}{Theorem}

\begin{document}
\global\long\def\diam{\operatorname{diam}}%

\global\long\def\pty{\operatorname{pty}}%

\global\long\def\Col{\operatorname{Col}}%

\global\long\def\cip{\operatorname{cir}}%

\global\long\def\sys{\operatorname{sys}}%

\global\long\def\Cay{\operatorname{Cay}}%

\global\long\def\Aut{\operatorname{Aut}}%

\global\long\def\R{\mathbb{R}}%

\global\long\def\vv{\rule{.4pt}{1ex}}%

\begin{lyxgreyedout}
\end{lyxgreyedout}

\title[Invariant Synchrony Subspaces]{Invariant Synchrony Subspaces of Sets of Matrices}
\author{John M. Neuberger, N\'andor Sieben, James W. Swift}
\curraddr{Northern Arizona University, Department of Mathematics and Statistics,
Flagstaff, AZ 86011-5717, USA}
\email{john.neuberger@nau.edu}
\email{nandor.sieben@nau.edu}
\email{jim.swift@nau.edu}
\thanks{Date: \the\month/\the\day/\the\year}
\begin{abstract}
A synchrony subspace of $\R^{n}$ is defined by setting certain components
of the vectors equal according to an equivalence relation. Synchrony
subspaces invariant under a given set of square matrices ordered by
inclusion form a lattice. Applications of these invariant synchrony
subspaces include equitable and almost equitable partitions of the
vertices of a graph used in many areas of graph theory, balanced and
exo-balanced partitions of coupled cell networks, and coset partitions
of Cayley graphs. We study the basic properties of invariant synchrony
subspaces and provide many examples of the applications. We also present
what we call the\textcolor{blue}{{} }\textcolor{black}{split and cir}
algorithm for finding the lattice of invariant synchrony subspaces.
Our theory and algorithm is further generalized for non-square matrices.
This leads to the notion of tactical decompositions studied for its
application in design theory.

\end{abstract}

\maketitle

\section{Introduction}

Invariant subspaces of matrices play an important role in many areas
of mathematics \cite{invariant}. We study a special type of invariant
subspace, the invariant synchrony subspace where certain components
are equal, and further generalize this notion by considering invariance
with respect to more than one matrix. These invariant synchrony subspaces
come from certain partitions of the components, which we call invariant
partitions. Invariant partitions have a surprisingly large number
of applications and have been studied, often in different form, in
graph theory, network theory, mathematical biology,\textcolor{blue}{{}
}and other areas of mathematics. For example, they relate to almost
equitable partitions, coupled cell networks, dynamical systems, network
controllability, and differential equations. Our own interest in the
subject initiated from needing to understand all possible bifurcations
for a given family of semilinear elliptic PDE, including cases where
symmetry did not fully describe the invariant spaces. In that research
it became apparent that what we called anomalous invariant subspaces
can occur quite frequently, and must be understood in order to robustly
and efficiently find and follow new branches of solutions with local
symmetry. Thus, it is desirable to be able to find all invariant partitions,
which form a lattice. Our main contribution is the \emph{split and
cir} algorithm that finds this lattice relatively well, by which we
mean it has some efficiencies that allow a modestly large calculation,
and for multiple matrices. A secondary goal of this article is to
provide a common framework that brings together different areas of
research so that we can better share results.

There are a number of papers that describe algorithms for finding
the coarsest invariant partitions for a single matrix \cite{Aldis,Bastert,Belykh}.
These algorithms were developed for algebraic graph theory and are
often described in graph theoretic terminology (e.g., vertex degree).
For example, this is used for the graph isomorphism problem in Nauty
\cite{nauty}.

Our cir algorithm is a slight generalization of these algorithms.
It finds the coarsest invariant refinement of a given partition for
a set of matrices. The coarsest invariant partition is the cir of
the singleton partition. Our implementation of cir was inspired by
\cite{Godsil,ZhangCamlibelCao}. By recursively splitting classes
of an invariant partition and finding the cir of each of these finer
partitions, we compute the lattice of invariant partitions. This provides
us with a relatively efficient split and cir algorithm for the NP-complete
problem of finding the lattice of invariant partitions. The recursive
combination of cir with class splitting is the main contribution of
our paper. 

Our algorithm is close to that of \cite{KameiLattice}. They find
the lattice of invariant partitions by finding the coarsest invariant
partition and then check every finer partition by brute force. We
believe that our split and cir algorithm with the recursive application
of cir is more efficient, allowing for larger calculations.

By quite different means toward a similar end, the cell network results
of \cite{Aguiar&Dias} use eigenvectors of each individual adjacency
matrix after dividing the network into regular subnetworks. We believe
that our algorithm is more numerically stable since it uses integer
arithmetic.

We generalize our split and cir algorithm a bit further to allow non-square
matrices. This handles a possibly generalized notion (more than one
matrix) of tactical decompositions used in design theory.

The organization of the paper is as follows. In Section~2 we make
some definitions and examples concerning the partial ordering of partitions.
We define synchrony subspaces in Section~3, where we give some useful
facts concerning the lattice of partitions. In Section~4 we define
invariant synchrony subspaces for collections of matrices, and show
that invariant partitions form a lattice. Applications of invariant
synchrony subspaces are provided in Section~5. We feature equitable
and almost equitable partitions of graphs, balanced and exo-balanced
partitions of coupled cell networks, weighted cell network systems,
network controllability, Cayley graphs, and finite difference methods
for PDE. The invariant subspaces for all of these applications can
be treated in terms of invariant partitions for the appropriate collection
of matrices. Section~6 provides some technical tools needed for the
cir algorithm, which is developed in Section~7. We provide the split
and cir algorithm for computing the lattice of invariant partitions
in Section~8. We parallel this development for non-square matrices
with applications to tactical decompositions in Sections~9, 10, and
11.

\section{Preliminaries}

We recall some terminology and notation for partially ordered sets.
Our general reference is \cite{Gratzer}. A set $P$ with a relation
$\le$ is a \emph{partially ordered set} or \emph{poset} if the relation
is reflexive, antisymmetric, and transitive. The \emph{down-set} of
an element $x$ of $P$ is 
\[
\downarrow x:=\{y\in P\mid y\le x\}.
\]

A poset $L$ is a \emph{lattice} if the \emph{join} $a\vee b:=\sup\{a,b\}$
and \emph{meet} $a\wedge b:=\inf\{a,b\}$ exist for all $a,b\in L$.
Equivalently, $L$ is a lattice if $\bigvee H:=\sup(H)$ and $\bigwedge H:=\inf(H)$
exist for any finite subset $H$ of $L$. A lattice $L$ is called
\emph{complete} if $\bigvee H$ and $\bigwedge H$ exist for any subset
$H$ of $L$.

An element $y$ of a poset \emph{covers} another element $x$ if $x<y$
but $x<z<y$ does not hold for any $z$. We use the notation $x\prec y$
for the covering relation. We also say that $x$ is a \emph{lower
cover} of $y$ and $y$ is an \emph{(upper) cover} of $x$.
\begin{example}
The set $L(\mathbb{R}^{n})$ of subspaces of $\mathbb{R}^{n}$ is
a poset under reversed containment. It is also a lattice such that
$U\vee V=U\cap V$ and $U\wedge V=U+V$.
\end{example}

\begin{example}
Let $\Pi(n)$ be the set of partitions of the set $\{1,\ldots,n\}$.
We are going to simply write $\Pi$ if $n$ is clear from the context.
We refer to an element of a partition as an (equivalence) class. The
class of $i$ is denoted by $[i]$. A partition $\mathcal{B}$ is
\emph{coarser} than another partition $\mathcal{A}$ if every class
of $\mathcal{A}$ is contained in some class of $\mathcal{B}$. Equivalently,
$\mathcal{B}$ is coarser than $\mathcal{A}$ exactly when $[i]_{\mathcal{A}}=[j]_{\mathcal{A}}$
implies $[i]_{\mathcal{B}}=[j]_{\mathcal{B}}$ for all $i$ and $j$.
In this case we write $\mathcal{A}\le\mathcal{B}$ and also say that
$\mathcal{A}$ is \emph{finer} than $\mathcal{B}$ or that $\mathcal{A}$
is a \emph{refinement} of $\mathcal{B}$. The set $\Pi(n)$ with this
order forms a lattice \cite[Theorem 4.11]{universalAlgebra}. The
top element of this lattice is the singleton partition $\{\{1,\ldots,n\}\}$
and the bottom element is the discrete partition $\{\{1\},\ldots,\{n\}\}$.
We use the simplified notation for a partition where we list the elements
of the classes separated by a bar. So the singleton partition is $12\cdots n$
and the discrete partition is $1|2|\cdots|n$.
\end{example}

A partition $\mathcal{B}$ covers another partition $\mathcal{A}$
if $\mathcal{A}$ can be constructed from $\mathcal{B}$ by splitting
one of the classes of $\mathcal{B}$ into two nonempty sets. More
precisely, we have the following.
\begin{prop}
\label{prop:spliting}For $\mathcal{A},\mathcal{B}\in\Pi(n)$, $\mathcal{A}\prec\mathcal{B}$
if and only if $\mathcal{B}=\mathcal{A}\cup\{A\cup B\}\setminus\{A,B\}$
for some $A,B\in\mathcal{A}$.
\end{prop}

\section{\label{sec:Synchrony-subspaces}Synchrony subspaces}

The \emph{characteristic vector} of a subset $C$ of $\{1,\ldots,n\}$
is the vector $(x_{1},\ldots,x_{n})$ such that $x_{i}=1$ for $i\in C$
and $x_{i}=0$ for $i\not\in C$. The \emph{characteristic matrix}
of a partition $\mathcal{A}=\{C_{1},\ldots,C_{k}\}\in\Pi(n)$ is the
matrix $P(\mathcal{A})$ whose columns are the characteristic vectors
of the classes of $\mathcal{A}$. The classes of $\mathcal{A}$ are
ordered lexicographically and the characteristic vectors are listed
in the same order. The \emph{coloring vector} of $\mathcal{A}$ is
$c(\mathcal{A}):=(c_{1},\ldots,c_{n})$ such that $i\in C_{c_{i}}$
for all $i\in V$. Note that $P(\mathcal{A})$ is the $n\times|\mathcal{A}|$
matrix with components $P(\mathcal{A})_{i,a}=\delta_{c_{i},a}$. We
use the identification $\mathbb{R}^{n}\equiv\mathbb{R}^{n\times1}$,
so the coloring vector $c(\mathcal{A}_{k})$ is also a column matrix.

The set of all coloring vectors is linearly ordered by the lexicographic
order. If $\mathcal{A\prec\mathcal{B}}$ then $c(\mathcal{A})<c(\mathcal{B})$
since one of the colors in the class that splits must increase. Hence
the coloring vector map is order preserving. That is, $\mathcal{A}\le\mathcal{B}$
implies $c(\mathcal{A})\le c(\mathcal{B})$. However the converse
is not true. For example $c(12|3)=(1,1,2)<(1,2,2)=c(1|23)$ even though
the partitions $12|3$ and $1|23$ are not comparable. We usually
order a list of partitions according to the lexicographical order
of their coloring vectors. Thus, the first partition in $\Pi(n)$
is the singleton partition $\mathcal{A}_{1}$ with coloring vector
$(1,1,1,\ldots,1)$ and the last partition is the discrete partition
$\mathcal{A}_{B_{n}}$ with coloring vector $(1,2,3,\ldots,n)$, where
$B_{n}$ is the $n$-th Bell number.

\begin{defn}
For $\mathcal{A}\in\Pi(n)$ we define
\[
\sys(\mathcal{A}):=\{(x_{1},\ldots,x_{n})\mid x_{i}=x_{j}\text{ if }[i]=[j]\}\subseteq\mathbb{R}^{n}
\]
to be the \emph{synchrony subspace} of $\mathcal{A}$.
\end{defn}

The connection between $\sys(\mathcal{A})$ and $\Delta_{\mathcal{A}}^{P}$
defined for coupled cell networks is discussed in Subsection~\ref{subsec:Balanced}.
The following is an easy consequence of the definitions.
\begin{prop}
The \emph{synchrony subspace} $\sys(\mathcal{A})$ of a partition
$\mathcal{A}$ is the column space $\Col(P(\mathcal{A}))$ of the
characteristic matrix $P(\mathcal{A})$.
\end{prop}

\begin{example}
Consider the partitions $\mathcal{A}=12|3$ and $\mathcal{B}=1|23$.
The characteristic matrices are
\[
P(\mathcal{A})=\left[\begin{smallmatrix}1 & 0\\
1 & 0\\
0 & 1
\end{smallmatrix}\right],\qquad P(\mathcal{B})=\left[\begin{smallmatrix}1 & 0\\
0 & 1\\
0 & 1
\end{smallmatrix}\right],
\]
the coloring vectors are $c(\mathcal{A})=(1,1,2)$ and $c(\mathcal{B})=(1,2,2)$
while the synchrony subspaces are $\sys(\mathcal{A})=\{(a,a,b)\mid a,b\in\mathbb{R}\}$
and $\sys(\mathcal{B})=\{(a,b,b)\mid a,b\in\mathbb{R}\}$.
\end{example}

The following is an easy consequence of the definitions. See for example
\cite[Equation (4)]{ZhangMingKanat}.
\begin{prop}
The mapping $\sys:\Pi(n)\to L(\mathbb{R}^{n})$ is an order embedding.
That is, it is injective and $\mathcal{A}\le\mathcal{B}$ if and only
if $\sys(\mathcal{A})\supseteq\sys(\mathcal{B})$.
\end{prop}

We recall another relevant result.
\begin{prop}
\label{prop:vee}\cite[Lemma 2]{ZhangCamlibelCao} If $S\subseteq\Pi(n)$
then $\sys(\bigvee S)=\bigcap\{\sys(\mathcal{A})\mid\mathcal{A}\in S\}$.
\end{prop}

\section{Invariant synchrony subspaces}

A subspace $U$ of $\mathbb{R}^{n}$ is \emph{invariant under a matrix}
$M$ in $\mathbb{R}^{n\times n}$ if $MU=\{Mu\mid u\in U\}$ is a
subset of $U$. The subspace $U$ is \emph{invariant under a collection}
$\mathcal{M}$ of matrices if $U$ is invariant under each $M$ in
$\mathcal{M}$.
\begin{defn}
Let $\mathcal{M}\subseteq\mathbb{R}^{n\times n}$. A partition $\mathcal{A}$
in $\Pi(n)$ is called $\mathcal{M}$-\emph{invariant} if the synchrony
subspace of $\mathcal{A}$ is $M$-invariant for all $M\in\mathcal{M}$.
That is, $M\sys(\mathcal{A})\subseteq\sys(\mathcal{A})$ for all $M\in\mathcal{M}$.
We let $\Pi_{\mathcal{M}}(n)$ be the set of $\mathcal{M}$-invariant
partitions. We use the simplified notation $\Pi_{M}(n)$ if $\mathcal{M}=\{M\}$.
\end{defn}

Note that $\Pi_{\mathcal{M}}(n)$ is a subposet of $\Pi(n)$.
\begin{example}
Let 
\[
M_{1}=\left[\begin{smallmatrix}1 & 0 & 0\\
0 & 1 & 0\\
0 & 0 & 2
\end{smallmatrix}\right],\quad M_{2}=\left[\begin{smallmatrix}1 & 0 & -1\\
0 & 2 & 0\\
0 & 0 & 0
\end{smallmatrix}\right]
\]
and $\mathcal{M}=\{M_{1},M_{2}\}$. Then
\[
\begin{aligned}\Pi_{M_{1}}(3) & =\{12|3,1|2|3\},\\
\Pi_{M_{2}}(3) & =\{13|2,1|2|3\},\\
\Pi_{\mathcal{M}}(3) & =\{1|2|3\}.
\end{aligned}
\]
For example 
\[
M_{2}\sys(13|2)=\{M_{2}(a,b,a)\mid a,b\in\mathbb{R}\}=\{(0,2b,0)\mid a,b\in\mathbb{R}\}\subseteq\sys(13|2).
\]
On the other hand
\[
M_{2}\sys(12|3)=\{M_{2}(a,a,b)\mid a,b\in\mathbb{R}\}=\{(a-b,2a,0)\mid a,b\in\mathbb{R}\}\not\subseteq\sys(12|3).
\]
\end{example}

\begin{prop}
\label{prop:supInvariant}If $\emptyset\ne S\subseteq\Pi_{\mathcal{M}}$
then $\bigvee S\in\Pi_{\mathcal{M}}$.
\end{prop}

\begin{proof}
Since $\sys(\mathcal{A})$ is $\mathcal{M}$-invariant for all $\mathcal{A}\in S$,
$\sys(\bigvee S)=\bigcap\{\sys(\mathcal{A})\mid\mathcal{A}\in S\}$
is also $\mathcal{M}$-invariant.
\end{proof}
Recall the following result.
\begin{prop}
\label{prop:latticeCondition}\cite[Lemma 34]{Gratzer} If $P$ is
a poset in which $\bigvee S$ exists for all $S\subseteq P$, then
$P$ is a complete lattice.
\end{prop}

Note that the bottom element of $P$ is $\bigvee\emptyset$, and $\bigwedge S=\bigvee L_{S}$
where $L_{S}$ is the set of lower bounds of $S$. The following is
closely related to \cite[Section 4]{Aguiar&Dias}.
\begin{prop}
If $\mathcal{M}\subseteq\mathbb{R}^{n\times n}$ then $\Pi{}_{\mathcal{M}}(n)$
is a lattice.
\end{prop}

\begin{proof}
We verify the conditions of Proposition~\ref{prop:latticeCondition}.
The discrete partition is invariant under any $\mathcal{M}$, so it
is the bottom element of $\Pi_{\mathcal{M}}(n)$. Let $\emptyset\ne S\subseteq\Pi_{\mathcal{M}}(n)$
and $\bigvee S$ be the supremum of $S$ taken in the lattice $\Pi(n)$.
Then $\bigvee S\in\Pi_{\mathcal{M}}(n)$ by Proposition~\ref{prop:supInvariant}.
So the supremum of $S$ in $\Pi_{\mathcal{M}}(n)$ is $\bigvee S$.
\end{proof}
\begin{figure}
\begin{tabular}{cc}
$\left[\begin{smallmatrix}1 & 1 & 0 & 0 & 0\\
1 & 1 & 0 & 0 & 0\\
1 & 1 & 0 & 0 & 0\\
0 & 0 & 1 & 1 & 0\\
0 & 0 & 0 & 1 & 1
\end{smallmatrix}\right]$ & \tabularnewline
 & \tabularnewline
\end{tabular}%
\begin{tabular}{lll}
$\mathcal{A}_{1}=12345$ & $\mathcal{A}_{7}=13|2|4|5$ & \tabularnewline
$\mathcal{A}_{2}=1234|5$ & $\mathcal{A}_{8}=14|235$ & \tabularnewline
$\mathcal{A}_{3}=123|4|5$ & $\mathcal{A}_{9}=14|23|5$ & \tabularnewline
$\mathcal{A}_{4}=12|3|4|5$ & $\mathcal{A}_{10}=1|23|4|5$ & \tabularnewline
$\mathcal{A}_{5}=135|24$ & $\mathcal{A}_{11}=1|2|3|4|5$ & \tabularnewline
$\mathcal{A}_{6}=13|24|5$ &  & \tabularnewline
\end{tabular}%
\begin{tabular}{c}
\includegraphics{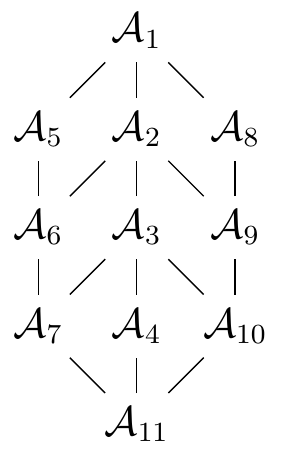}\tabularnewline
\end{tabular}

\caption{\label{fig:sublattice}A matrix $M$ and the lattice of $M$-invariant
partitions.}
\end{figure}

\begin{figure}
\begin{tabular}{cccc}
\begin{tabular}{c}
\includegraphics{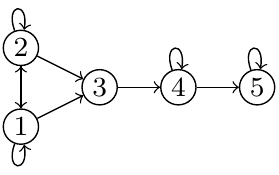}\tabularnewline
\end{tabular} & %
\begin{tabular}{c}
\includegraphics{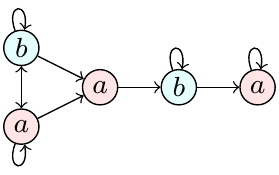}\tabularnewline
\end{tabular} & %
\begin{tabular}{c}
\includegraphics{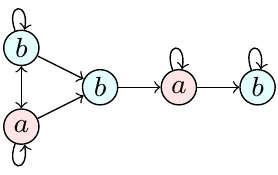}\tabularnewline
\end{tabular} & %
\begin{tabular}{c}
\includegraphics{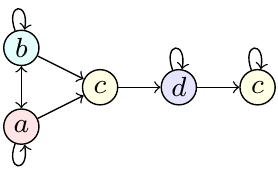}\tabularnewline
\end{tabular}\tabularnewline
 & $\mathcal{A}_{5}$ & $\mathcal{A}_{8}$ & $\mathcal{A}_{5}\wedge\mathcal{A}_{8}$\tabularnewline
 & balanced & balanced & not balanced\tabularnewline
\end{tabular}

\caption{\label{fig:infNotBalanced}Balanced partitions whose infimum in $\Pi$
is not balanced.}
\end{figure}

\begin{example}
\label{exa:sublattice}Figure~\ref{fig:sublattice} shows a matrix
$M$ and the lattice of $M$-invariant partitions. Partitions $\mathcal{A}_{5}=135|24$
and $\mathcal{A}_{8}=14|235$ are $M$-invariant but $\mathcal{A}_{5}\wedge\mathcal{A}_{8}=1|2|35|4$
is not $M$-invariant. The infimum of $\mathcal{A}_{5}$ and $\mathcal{A}_{8}$
in $\Pi_{M}$ is the discrete partition $\mathcal{A}_{11}$. This
shows that $\Pi_{\mathcal{M}}$ is not a sublattice of $\Pi$. Although
the $\vee$ operation is the same in $\Pi$ and $\Pi_{\mathcal{M}}$,
the $\wedge$ operations are not the same. See \cite[page 951]{Aguiar&Dias}
and \cite{StewartLattice}. Note that 
\[
\sys(\mathcal{A}_{5})+\sys(\mathcal{A}_{8})=\{x\in\R^{5}\mid x_{3}=x_{5},\,x_{1}+x_{2}=x_{3}+x_{4}\}
\]
 is not a synchrony subspace.
\end{example}

\section{Applications}

We present several areas of mathematics where invariant synchrony
subspaces are used. We introduce the concepts of equitable and almost
equitable partitions of graphs. We also talk about balanced and exo-balanced
partitions. These are essentially generalizations for colored digraphs.
The terminology is still evolving and not completely consistent but
generally equitable is used in graph theory and balanced is used for
network theory. All of these concepts can be treated as $\mathcal{M}$-invariant
partitions for an appropriate $\mathcal{M}$.

\subsection{Equitable partitions of graphs}

Equitable partitions play an important role in graph theory \cite{Godsil,GodsilSurvey,McKayMAthesis}.
They are used for example in algorithms to test if two graphs are
isomorphic \cite{nauty}. Orbits of group actions on graphs form equitable
partitions but not all equitable partitions are of this type. Equitable
partitions can be used to create quotient digraphs.

If $i$ is a vertex of a digraph, then the \emph{neighborhood} of
$i$ is denoted by $N(i)$. The neighborhood of $i$ consists of the
vertices adjacent to $i$. Let $\mathcal{A}$ be a partition of the
vertex set of a digraph. The \emph{degree} of a vertex $i$ \emph{relative
to the class} $B$ is $d_{B}(i):=|N(i)\cap B|$.
\begin{defn}
\label{def:equitable}A partition $\mathcal{A}$ of the vertex set
of a graph is called \emph{equitable} if $d_{B}(i)=d_{B}(j)$ for
all $i,j\in A\in\mathcal{A}$ and $B\in\mathcal{A}$.
\end{defn}

\begin{example}
Every partition of the vertex set of the complete graph $K_{n}$ is
equitable since $d_{A}(i)=|N(i)\cap A|=|A|-$1 for all $i\in A$ and
$d_{B}(i)=|N(i)\cap B|=|B|$ for all $i\in A\ne B$. So the number
of equitable partitions of $K_{n}$ is the Bell number $B_{n}$, which
is defined to be the number or partitions of a set with $n$ elements.
\end{example}

\begin{example}
If $A$ is the adjacency matrix of the graph with vertex set $\{1,\ldots,n\}$,
then $\Pi_{A}(n)$ consists of the equitable partitions of the graph
\cite[Lemma 2.1]{Godsil}.
\end{example}

We are going to refer to a directed graph with possible multiple edges
and loops as a \emph{digraph}.
\begin{example}
\begin{figure}
\begin{tabular}{ccccc}
\begin{tabular}{c}
\includegraphics{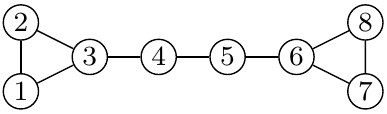}\tabularnewline
\end{tabular} & $\quad$ & %
\begin{tabular}{c}
\includegraphics{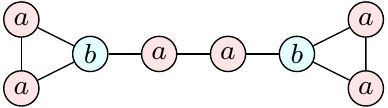}\tabularnewline
\end{tabular} & $\quad$ & %
\begin{tabular}{c}
\includegraphics{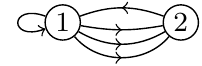}\tabularnewline
\end{tabular}\tabularnewline
$G$ &  & $\mathcal{A}_{1}=124578|36$ &  & $G/\mathcal{A}_{1}$\tabularnewline
(i) &  & (ii) &  & \multicolumn{1}{c}{(iii)}\tabularnewline
\end{tabular}

\caption{\label{fig:equitable}(i) Graph $G$. (ii) The coarsest equitable
partition $\mathcal{A}_{1}$ of $G$. (iii) The quotient digraph $G/\mathcal{A}_{1}$.}
\end{figure}
Figure~\ref{fig:equitable}(ii) shows the coarsest of the eight equitable
partitions of the graph $G$ shown in Figure~\ref{fig:equitable}(i).
Figure~\ref{fig:equitable}(iii) shows the quotient digraph by $\mathcal{A}_{1}$.
It is obvious that all orbit partitions under the graph automorphism
group are equitable, but the converse is not true. The equitable partition
of Figure~\ref{fig:equitable}(ii) is not an orbit partition. The
coarsest orbit partition of $G$ is $1278|36|45$. The full list of
equitable partitions for this so-called McKay graph \cite[Fig. 5.1]{McKayMAthesis}
are listed at the companion web site \cite{invariantWEB}.
\end{example}

Note that the coarsest orbit partition of a graph $G$ is always the
orbit partition for the natural action of $\Aut(G)$. Furthermore,
if $\Aut(G)$ has a vertex with trivial stabilizer subgroup, then
the group orbit of that vertex has the same size as that of $\Aut(G)$,
and every subgroup of $\Aut(G)$ has a distinct orbit partition. For
the graph in Figure~\ref{fig:equitable} there is no vertex with
trivial stabilizer subgroup, but such a vertex exists for the graphs
in the next example.
\begin{example}
\begin{figure}
\begin{tabular}{ccccc}
\begin{tabular}{c}
\includegraphics{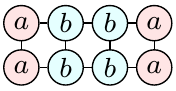}\tabularnewline
\end{tabular} & $\quad$ & %
\begin{tabular}{c}
\includegraphics{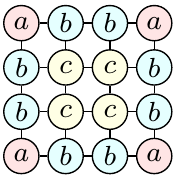}\tabularnewline
\end{tabular} & $\quad$ & %
\begin{tabular}{c}
\includegraphics{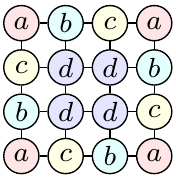}\tabularnewline
\end{tabular}\tabularnewline
(i) &  & (ii) &  & \multicolumn{1}{c}{(iii)}\tabularnewline
\end{tabular}

\caption{\label{fig:equitable-mxn}Equitable partitions. (i) The coarsest equitable
partition of $P_{4}\oblong P_{2}$. (ii) The coarsest equitable partition
of $P_{4}\oblong P_{4}$ . (iii) The orbit partition of the $\mathbb{Z}_{4}$
action on \textcolor{blue}{${\normalcolor P_{4}\oblong P_{4}}$.}}
\end{figure}
Figure~\ref{fig:equitable-mxn} shows a few equitable partitions
of the $m\times n$ square grid graphs $P_{m}\oblong P_{n}$. We have
computed the equitable partitions up to $m=n=20$, along with some
larger graphs. The companion web site \cite{invariantWEB} shows the
equitable partitions for several sizes. Unlike the graph in Figure~\ref{fig:equitable},
it appears that all of the equitable partitions of $P_{m}\square P_{n}$
are orbit partitions. 
\end{example}

\begin{conjecture}
\label{conj:PmPn}The equitable partitions of $P_{m}\oblong P_{n}$
are the orbit partitions of natural actions of the subgroups of the
automorphism group of the graph. The lattice of equitable partitions
is isomorphic to the lattice of subgroups exactly when there is a
vertex with trivial stabilizer subgroup. This happens unless $2\le m=n\le3$.
\end{conjecture}

Note that the orbit partitions for $\mathbb{Z}_{4}$ and $D_{4}$
are the same for $P_{2}\oblong P_{2}$ and $P_{3}\oblong P_{3}$.
An algorithm using the eigenvectors of the adjacency matrix \cite{Aguiar&Dias,KameiLattice}
can possibly be applied symbolically for all $m$ and $n$ to prove
this conjecture. Alternatively, the techniques of \cite{Antoneli2}
might also provide a proof.

\subsection{\label{subsec:Balanced}Balanced partitions of coupled cell networks\label{subsec:balanced-coupled-cell-networks}}

Synchrony in coupled cell networks introduced by \cite{Pivato} and
refined by \cite{Torok} provide examples of $\mathcal{M}$-invariant
synchrony subspaces. A \emph{coupled cell network system} is a set
of dynamical systems coupled according to the arrows of a digraph.
An interesting case is where some or all of the cells have the same
internal dynamics, so partial or full synchronization of the cells
is possible. Coupled cells are of interest to physicists and biologists
\cite{BonaccorsiOttavianoMugnoloPellegrini,PecoraClusterSynch,SchaubClusterSynch,SorrentinoClusterSynch,Rai}.
A good mathematical review is \cite{aldis2010balance}.

A \emph{coupled cell network} is formally defined in \cite{Torok}
as a digraph with an equivalence relation on the set $\{1,\ldots,n\}$
of cells and another equivalence relation on the arrows, along with
the consistency condition that equivalent arrows have equivalent heads
and equivalent tails. Intuitively, two cells are in the same equivalence
class if they have the same internal dynamics, and two arrows are
in the same equivalence class if they represent couplings of the same
type. The equivalence classes of the cells and arrows can be considered
as colorings. So a cell network is essentially a digraph with compatible
cell and arrow colorings. We use different symbols in our figures
to indicate the cell colors, and solid or dashed arrows for the edge
colors.

A coupled cell network system defined in \cite{Torok} on the coupled
cell network is an ODE defined by an \emph{admissible vector field}.
To define admissible vector fields, a phase space $P_{i}$ is specified
for each cell, and the total phase space is the Cartesian product
$P=P_{1}\times\cdots\times P_{n}$ of the cell phase spaces. Given
a coupled cell network on the digraph $G$ and a total phase space
$P$, the set of all admissible vector fields is $\mathcal{F}_{G}^{P}$,
the collection of all vector fields compatible with the structure
of the colored graph $G$.

The \emph{polydiagonal subspace} of a cell partition $\mathcal{A}$
of a cell network system is $\Delta_{\mathcal{A}}^{P}:=\{x\in P\mid x_{i}=x_{j}\mbox{ if }[i]=[j]\}$,\textcolor{blue}{{}
}see \cite{Aguiar&Dias,Torok,KameiLattice,NSS6,Pivato}\textcolor{blue}{.
}Note that $\Delta_{\mathcal{A}}^{P}=\sys(\mathcal{A})$ if $P=\R^{n}$.
A cell partition $\mathcal{A}$ of a coupled cell network system is
\emph{polysynchronous} if $\Delta_{\mathcal{A}}^{P}$ is invariant
under the vector field of the system. A cell partition $\mathcal{A}$
of a coupled cell network is \emph{robustly polysynchronous} if $\Delta_{\mathcal{A}}^{P}$
is invariant under every vector field in $\mathcal{F}_{G}^{P}$ \textcolor{black}{for
all choices of $P$}\textcolor{blue}{.}

The \emph{adjacency matrix} of a digraph with $n$ vertices is the
$n\times n$ matrix $A$ with $A_{i,j}$ equal to the number of arrows
from vertex $j$ to vertex $i$. Our adjacency matrix is the \emph{in-adjacency}
matrix, which is the preferred choice for cell networks \cite{Aguiar&Dias}.
In graph theory it is more common to use the \emph{out-adjacency}
matrix, which is the transpose of our $A$.
\begin{example}
The matrix of Example~\ref{exa:sublattice} shown in Figure~\ref{fig:sublattice}
is the adjacency matrix of the graph shown in Figure~\ref{fig:infNotBalanced}.
The adjacency matrix of the digraph in Figure~\ref{fig:equitable}(iii)
is $\left[\begin{smallmatrix}1 & 1\\
3 & 0
\end{smallmatrix}\right]$.
\end{example}

Restricting the cell network to a single edge type creates a \emph{monochrome
digraph} of the cell network. The information about a cell network
is encoded by the set $\mathcal{M}$ of adjacency matrices of these
monochrome digraphs together with the \emph{cell type partition} $\mathcal{T}$
determined by the vertex types.

If $i$ is a vertex of a digraph, then the \emph{in-neighborhood}
of $i$ is denoted by $N^{-}(i)$. The in-neighborhood consists of
the vertices from which there is an arrow to $i$. The in-neighborhood
is called the \emph{input set} in \cite{Pivato} and subsequent articles.
Let $\mathcal{A}$ be a partition of the vertex set of a digraph.
The \emph{in-degree} of a vertex $i$ \emph{relative to the class}
$B$ is $d_{B}^{-}(i):=|N^{-}(i)\cap B|$.
\begin{defn}
Consider a coupled cell network with cell type partition $\mathcal{T}$.
A cell partition $\mathcal{A}$ is called \emph{balanced} if $\mathcal{A}\le\mathcal{T}$
and $d_{B}^{-}(i)=d_{B}^{-}(j)$ for all $i,j\in A\in\mathcal{A}$
and $B\in\mathcal{A}$ in every monochrome subgraph.
\end{defn}

\begin{example}
A simple graph has no multiple edges or loops. Such a graph determines
a cell network with a single class of vertices and a single class
of edges. The adjacency matrix of the cell network is symmetric, in
fact, it is the adjacency matrix of the original graph. So the equitable
partitions of the graph are exactly the balanced partitions of the
corresponding cell network.
\end{example}

The following result is the main reason why balanced partitions are
important for cell networks.
\begin{prop}
\label{prop:mainReason}\cite[Theorem 4.3]{Torok}\textcolor{black}{{}
A cell partition of a coupled cell network is robustly polysynchronous
if and only if it is balanced.}
\end{prop}

Note that the choice of the phase space $P$ has no effect on whether
a partition is robustly polysynchronous, since being balanced does
not depend on $P$. The following is a well-known result, see \cite[Remark 2.12]{Aguiar&Dias}.
\begin{prop}
\textcolor{black}{\label{prop:balancedIsMinvariant}The set of balanced
partitions of a coupled cell network with set of adjacency matrices
$\mathcal{M}$ and cell type partition $\mathcal{T}$ is $\Pi_{\mathcal{M}}\cap\downarrow\mathcal{T}$.}
\end{prop}

\begin{rem}
Note that $\Pi_{\mathcal{M}}$ is computed from the arrow-colored
digraph, with no need for the cell type partition. I\textcolor{black}{f
$\mathcal{T}$ is the singleton partition, then $\downarrow\mathcal{T}=\Pi$
and so $\Pi_{\mathcal{M}}\cap\downarrow\mathcal{T}=\Pi_{\mathcal{M}}$.}
\end{rem}

\begin{example}
\label{exa:balanced}
\begin{figure}
\setlength{\tabcolsep}{1pt}

\begin{tabular}{ccccccccc}
\includegraphics{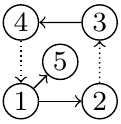} &  & \includegraphics{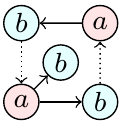} &  & \includegraphics{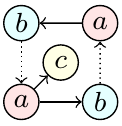} &  & \includegraphics{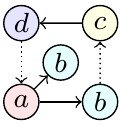} &  & \includegraphics{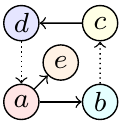}\tabularnewline
 & $\quad$ & $13|245$ & ~ & $13|24|5$ & ~ & $1|25|3|4$ & ~ & $1|2|3|4|5$\tabularnewline
 &  & $\mathcal{A}_{1}$ &  & $\mathcal{A}_{2}$ &  & $\mathcal{A}_{3}$ &  & $\mathcal{A}_{4}$\tabularnewline
(i) &  & \multicolumn{7}{c}{(ii)}\tabularnewline
\end{tabular}$\quad$%
\begin{tabular}{c}
\includegraphics{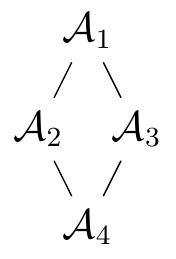}\tabularnewline
(iii)\tabularnewline
\end{tabular}

\caption{\label{fig:balEx}(i) A coupled cell network. (ii) Balanced partitions.
(iii) Lattice of balanced partitions.}
\end{figure}
Figure~\ref{fig:balEx} shows the balanced partitions of a cell network
with a single cell type. The cell type partition is $\mathcal{T}=12345$
and the set of adjacency matrices is $\mathcal{M}=\{M_{1},M_{2}\}$
where 
\[
M_{1}=\left[\begin{smallmatrix}0 & 0 & 0 & 0 & 0\\
1 & 0 & 0 & 0 & 0\\
0 & 0 & 0 & 0 & 0\\
0 & 0 & 1 & 0 & 0\\
1 & 0 & 0 & 0 & 0
\end{smallmatrix}\right],\quad M_{2}=\left[\begin{smallmatrix}\begin{smallmatrix}0 & 0 & 0 & 1 & 0\\
0 & 0 & 0 & 0 & 0\\
0 & 1 & 0 & 0 & 0\\
0 & 0 & 0 & 0 & 0\\
0 & 0 & 0 & 0 & 0
\end{smallmatrix}\end{smallmatrix}\right].
\]
The coarsest $\mathcal{M}$-invariant partition is $\mathcal{A}_{1}=13|245$.
This is the partition corresponding to input equivalence \cite{Torok},
since cells 1 and 3 feel one dashed arrow, and cells 2, 3, and 5 feel
one solid arrow.
\end{example}

In Example \ref{exa:balanced} the only cell type partitions compatible
with the arrow types are the singleton partition and $\mathcal{A}_{1}$,
the coarsest $\mathcal{M}$-invariant partition. In both cases, the
set of balanced partitions is the same. The next example shows that
sometimes the cell type partition $\mathcal{T}$ can effect the set
of balanced partitions. 
\begin{example}
\begin{figure}
\setlength{\tabcolsep}{1pt}

\begin{tabular}{ccccccccccccccc}
\begin{tabular}{c}
\includegraphics{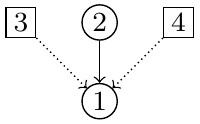}\tabularnewline
\end{tabular} &  & %
\begin{tabular}{c}
$1|234$\tabularnewline
$\mathcal{A}_{1}$\tabularnewline
\end{tabular} &  & %
\begin{tabular}{c}
$1|23|4$\tabularnewline
$\mathcal{A}_{2}$\tabularnewline
\end{tabular} &  & %
\begin{tabular}{c}
$1|24|3$\tabularnewline
$\mathcal{A}_{3}$\tabularnewline
\end{tabular} &  & %
\begin{tabular}{c}
$1|2|34$\tabularnewline
$\mathcal{A}_{4}$\tabularnewline
\end{tabular} &  & %
\begin{tabular}{c}
$1|2|3|4$\tabularnewline
$\mathcal{A}_{5}$\tabularnewline
\end{tabular} &  & %
\begin{tabular}{c}
\includegraphics{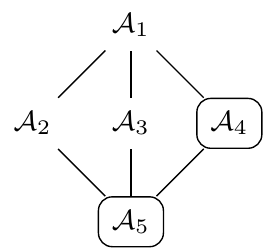}\tabularnewline
\end{tabular} &  & %
\begin{tabular}{c}
\includegraphics{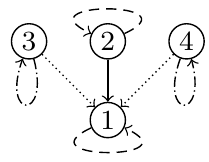}\tabularnewline
\end{tabular}\tabularnewline
(i) &  & \multicolumn{9}{c}{(ii)} &  & (iii) &  & (iv)\tabularnewline
\end{tabular}

\caption{\label{fig:balEx2}(i) A coupled cell network with two cell types
and two arrow types. (ii) $\mathcal{M}$-invariant partitions. (iii)
Lattice of $\mathcal{M}$-invariant partitions, with the balanced
partitions for the network circled. (iv) Another coupled cell network
with only one cell type and four arrow types\textcolor{blue}{,} and
the same balanced partitions as those of (i).}
\end{figure}
Consider the coupled cell network shown in Figure~\ref{fig:balEx2}(i).
The cell type partition is $\mathcal{T}=12|34$ and the set of adjacency
matrices is $\mathcal{M}=\{M_{1},M_{2}\}$, where 
\[
M_{1}=\left[\begin{smallmatrix}0 & 1 & 0 & 0\\
0 & 0 & 0 & 0\\
0 & 0 & 0 & 0\\
0 & 0 & 0 & 0
\end{smallmatrix}\right],\quad M_{2}=\left[\begin{smallmatrix}\begin{smallmatrix}0 & 0 & 1 & 1\\
0 & 0 & 0 & 0\\
0 & 0 & 0 & 0\\
0 & 0 & 0 & 0
\end{smallmatrix}\end{smallmatrix}\right].
\]

\noindent The elements of \textcolor{black}{$\Pi_{\mathcal{M}}(4)$}
are shown in Figure~\ref{fig:balEx2}(ii). The input partition, as
defined in \cite[Definition 2.3]{Torok}, is $\mathcal{A}_{1}$. Note
that the input partition $\mathcal{A}_{1}$ is not a refinement of
the cell type partition $\mathcal{T}$. This is the result of the
lack of incoming arrows at cells 2, 3, and 4. The balanced partitions
$\mathcal{A}_{4}$ and $\mathcal{A}_{5}$ of the cell network are
circled in Figure~\ref{fig:balEx2}(iii). As described by Proposition
\ref{prop:balancedIsMinvariant}, these  are the $\mathcal{M}$-invariant
partitions that are finer than $\mathcal{T}$.

Figure~\ref{fig:balEx2}(iv) shows another coupled cell network whose
balanced partitions are also $\mathcal{A}_{4}$ and $\mathcal{A}_{5}$.
There is only one vertex type in this network, so the cell type partition
is the singleton partition. The set of adjacency matrices for this
network is $\mathcal{M}_{\mathcal{T}}=\{M_{1},M_{2},V_{1},V_{2}\}$,
where 
\[
V_{1}=\left[\begin{smallmatrix}1 & 0 & 0 & 0\\
0 & 1 & 0 & 0\\
0 & 0 & 0 & 0\\
0 & 0 & 0 & 0
\end{smallmatrix}\right],\quad V_{2}=\left[\begin{smallmatrix}\begin{smallmatrix}0 & 0 & 0 & 0\\
0 & 0 & 0 & 0\\
0 & 0 & 1 & 0\\
0 & 0 & 0 & 1
\end{smallmatrix}\end{smallmatrix}\right].
\]
The loop structure of the second network takes over the role of the
cell type partition in the first network.
\end{example}

The previous example shows that the effect of the cell type partition
can be replaced by adding loop arrows to the cell network.
\begin{prop}
Let $\mathcal{T}$ be the cell type partition of a coupled cell network.
We construct a new coupled cell network by adding a loop at every
cell with arrow type identical to the cell type of the cell and changing
the cell type partition to be the singleton partition. The two coupled
cell networks have the same balanced partitions.
\end{prop}

\begin{proof}
Let $(c_{1},\ldots,c_{n}):=c(\mathcal{T})$ be the coloring vector
of the partition $\mathcal{T}$ of $\Pi(n)$. It is an easy consequence
of the definitions that $\mathcal{A}\le\mathcal{T}$ if and only if
$\mathcal{A}$ is a balanced partition of the $n$-cell network that
has a single loop arrow on cell $i$ with type $c_{i}$ for all $i$.
The result is an immediate consequence of this equivalence.
\end{proof}
\textcolor{black}{This result shows that robust polysynchrony of coupled
cell networks of \cite{Torok} can be fully described in terms of
a slightly modified arrow-colored digraph and its partitions that
are invariant under the set of adjacency matrices.}

\subsection{Cayley graphs}

Cayley graphs play an important role in group theory. There are several
types of Cayley graphs. Let $S$ be a subset of the group $\Gamma$.
The \emph{Cayley color digraph} $\Cay_{S}(\Gamma)$ of $\Gamma$ has
vertex set $\Gamma$ and arrow set $\{(g,gs)\mid g\in\Gamma,s\in S\}$.
The arrow $(g,gs)$ is colored by color $s$. If $S$ is a generating
set, then $\Cay_{S}(\Gamma)$ is connected, with an automorphism group
isomorphic to $\Gamma$. The following is clear from the definition.
\begin{lem}
\label{lem:nbd}The in-neighborhood of a vertex $g$ of $\Cay_{\{s\}}(\Gamma)$
is $N^{-}(g)=\{gs^{-1}\}$.
\end{lem}

The following is an immediate corollary.
\begin{cor}
\label{cor:inDeg}If $B$ is a class of a partition of $\Cay_{\{s\}}(\Gamma)$
and $g\in\Gamma$, then
\[
d_{B}^{-}(g)=\begin{cases}
0, & gs^{-1}\notin B\\
1, & gs^{-1}\in B.
\end{cases}
\]
\end{cor}

\begin{lem}
\label{lem:path}Consider a balanced partition of $\Cay_{\{s\}}(\Gamma)$.
If $[g]=[h]$ then $[gs^{-1}]=[hs^{-1}]$.
\end{lem}

\begin{proof}
If $[gs^{-1}]\not=[hs^{-1}]$ then 
\[
d_{[gs^{-1}]}^{-}(g)=1\ne0=d_{[gs^{-1}]}^{-}(h)
\]
by Corollary~\ref{cor:inDeg}, which is impossible in a balanced
partition.
\end{proof}
\begin{prop}
\label{prop:Cayley}Let $S$ be a generating set of a group finite
group $\Gamma$. The balanced partitions of $\Cay_{S}(\Gamma)$ are
the right coset partitions of $\Gamma$ by subgroups.
\end{prop}

\begin{proof}
First we show that the right coset partition of $\Gamma$ with respect
to the subgroup $H$ is a balanced partition of the vertices of $\Cay_{\{s\}}(\Gamma)$
for all $s\in S$. Let $g_{i}\in\Gamma$ and $Hg$ be a coset. Then
\[
d_{Hg}^{-}(g_{i})=\begin{cases}
0, & g_{i}s^{-1}\not\in Hg\\
1, & g_{i}s^{-1}\in Hg
\end{cases}=\begin{cases}
0, & g_{i}\not\in Hgs\\
1, & g_{i}\in Hgs
\end{cases}
\]
by Corollary~\ref{cor:inDeg}. So the in-degree of $g_{i}$ relative
to $Hg$ is only dependent on the coset that contains $g_{i}$.

Now consider a balanced partition of $\Cay_{S}(\Gamma)$. This partition
is balanced in $\Cay_{\{s\}}(\Gamma)$ for all $s\in S$. Since $\Cay_{S}(\Gamma)$
is vertex transitive, it suffices to show that the class $[e]$ of
the identity element $e$ of $\Gamma$ is a subgroup of $\Gamma$
by checking that if $g_{i}$ and $g_{j}$ are in $[e]$, then so is
$g_{i}^{-1}g_{j}$. Since $S$ is a generating set, there is a path
in $\Cay_{S}(\Gamma)$ from $e$ to $g_{j}$ along some edges colored
with $s_{1},\ldots,s_{k}$ respectively. Then $g=es_{1}\cdots s_{k}$,
so the path starting at vertex $g_{i}g_{j}^{-1}$ along the edges
colored with $s_{1},\ldots,s_{k}$ respectively leads to $g_{i}g_{j}^{-1}s_{1}\cdots s_{k}=g_{i}$.
Since $[g_{i}]=[g_{j}]$, the repeated application of Lemma~\ref{lem:path}
implies that $[e]=[g_{j}]=[g_{i}g_{j}^{-1}]$.
\end{proof}
\begin{figure}
\begin{tabular}{ccc}
\begin{tabular}{c}
\includegraphics{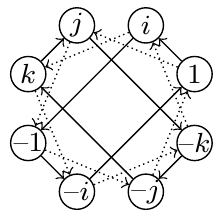}\tabularnewline
\includegraphics[clip]{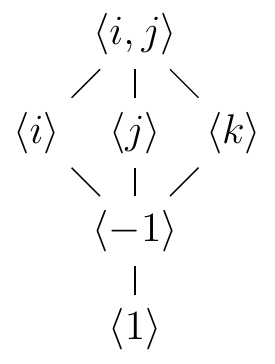}\tabularnewline
\end{tabular} & %
\begin{tabular}{ccc}
\includegraphics{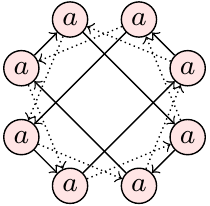} & \includegraphics{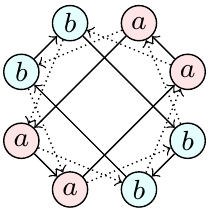} & \includegraphics{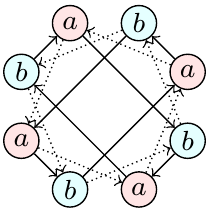}\tabularnewline
$12345678$ & $1256|3478$ & $1357|2468$\tabularnewline
$\mathcal{A}_{1}$ & $\mathcal{A}_{2}$ & $\mathcal{A}_{3}$\tabularnewline
\includegraphics{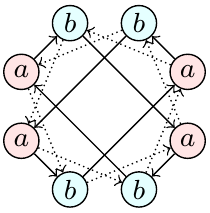} & \includegraphics{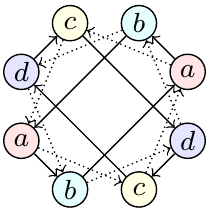} & \includegraphics{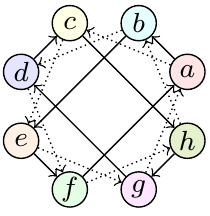}\tabularnewline
$1458|2367$ & $15|26|37|48$ & $1|2|3|4|5|6|7|8$\tabularnewline
$\mathcal{A}_{4}$ & $\mathcal{A}_{5}$ & $\mathcal{A}_{6}$\tabularnewline
\end{tabular} & %
\begin{tabular}{c}
\includegraphics{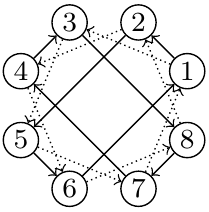}\tabularnewline
\includegraphics[clip]{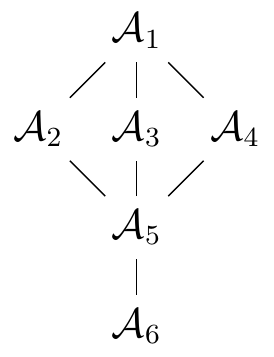}\tabularnewline
\end{tabular}\tabularnewline
(i) & (ii) & (iii)\tabularnewline
\end{tabular}

\caption{\label{fig:q8} (i) The Cayley color digraph $\protect\Cay_{\{i,j\}}(Q_{8})$
of the quaternion group $Q_{8}$ and its lattice of subgroups. A solid
arrow of the Cayley digraph is colored with $i$ while a dotted arrow
is colored with $j$. (ii) The balanced partitions of the cell network.
(iii) The corresponding cell network and the lattice of balanced partitions.}
\end{figure}

\begin{example}
Figure~\ref{fig:q8} illustrates Proposition~\ref{prop:Cayley}
for the Cayley color digraph $\Cay_{\{i,j\}}(Q_{8})$ of the quaternion
group, $Q_{8}=\{\pm1,\pm i,\pm j,\pm k\}$ with $i^{2}=j^{2}=k^{2}=ijk=-1$.
\end{example}

\subsection{Almost equitable partitions of graphs}

Almost equitable partitions were introduced in \cite{CardosoDelormeRama}
and used in \cite{BonaccorsiOttavianoMugnoloPellegrini,spectra,Gambuzza,Gerbaud,MonacoRicciardi}
and in many applications in network controllability, described in
Subsection~\ref{subsec:controllability}. It is very similar to the
notion of the equitable partition of Definition~\ref{def:equitable},
except only edges joining different sets in the partition are considered.
\begin{defn}
\label{def:almostEquitable}A partition $\mathcal{A}$ of the vertex
set of a graph is called \emph{almost equitable} if $d_{B}(i)=d_{B}(j)$
for all $i,j\in A\in\mathcal{A}$ and $A\ne B\in\mathcal{A}$.
\end{defn}

Almost equitable partitions are also known \cite{AguiarDiasWeighted,SchaubClusterSynch}
as external equitable partitions. Note that the singleton partition
is almost equitable and an equitable partition is almost equitable.
\begin{example}
\begin{figure}
\begin{tabular}{c}
\includegraphics{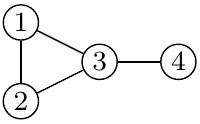}\tabularnewline
\end{tabular}%
\begin{tabular}{cccc}
$1234$ & $124|3$ & $12|3|4$ & $1|2|3|4$\tabularnewline
$\mathcal{A}_{1}$ & $\mathcal{A}_{2}$ & $\boxed{\mathcal{A}_{3}}$ & $\boxed{\mathcal{A}_{4}}$\tabularnewline
\end{tabular}

\caption{\label{fig:Paw}A graph and its almost equitable partitions. The boxed
partitions are equitable.}
\end{figure}

Figure~\ref{fig:Paw} shows a graph with its almost equitable partitions.
Only partitions $\mathcal{A}_{3}$ and $\mathcal{A}_{4}$ are equitable\textcolor{blue}{.
}The almost equitable partitions are linearly ordered since $\mathcal{A}_{1}>\mathcal{A}_{2}>\mathcal{A}_{3}>\mathcal{A}_{4}$.
\end{example}

The Laplacian matrix of a graph is defined by $L=D-A$, where $D$
is the diagonal matrix with $D_{i,i}=d(i)$ the degree of $i$, and
$A$ is the adjacency matrix. The following is a well-known result
stating that the $L$-invariant partitions are precisely the almost
equitable partitions. See for example \cite[Proposition 1]{CardosoDelormeRama},
\cite[Remark 2.2 (ii)]{AguiarDiasWeighted}, and \cite[Remark 3.14]{NSS6}.
\begin{prop}
If $L$ is the Laplacian matrix of a graph with vertex set $\{1,\ldots,n\}$,
then $\Pi_{L}(n)$ consists of the almost equitable partitions of
the graph.
\end{prop}

This proposition shows that almost equitable partitions of a graph
can also be viewed as invariant partitions.
\begin{example}
\begin{figure}
\begin{tabular}{ccccc}
\begin{tabular}{c}
\includegraphics{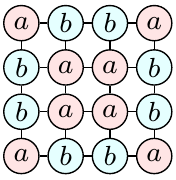}\tabularnewline
\end{tabular} & $\quad$ & %
\begin{tabular}{c}
\includegraphics{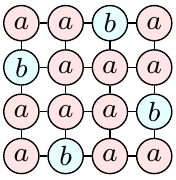}\tabularnewline
\end{tabular} & $\quad$ & %
\begin{tabular}{c}
\includegraphics{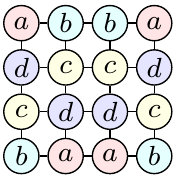}\tabularnewline
\end{tabular}\tabularnewline
(i) &  & (ii) &  & \multicolumn{1}{c}{(iii)}\tabularnewline
\end{tabular}

\caption{\label{fig:almosEquitable4x4}Three of the 13 almost equitable, but
not equitable, partitions of $P_{4}\square P_{4}$\textcolor{blue}{.}}
\end{figure}
The almost equitable partitions of $P_{m}\oblong P_{n}$ are much
more complicated than the equitable partitions. A conjectured classification
of the almost equitable partitions of $P_{1}\square P_{n}\cong P_{n}$
is in \cite{NSS6}\textcolor{blue}{.}\textcolor{red}{{} }The graph $P_{4}\oblong P_{4}$
has 23 almost equitable partitions in 17 group orbit classes. A few
of these partitions are shown in Figure~\ref{fig:almosEquitable4x4}.
The companion website \cite{invariantWEB} shows them all. By contrast
there are 10 equitable partitions in 8 group orbit classes for this
graph. Some almost equitable partitions for $P_{n}\square P_{n}$
are found in \cite{GillisGolubitsky} but not all of them. For example,
the almost equitable partitions shown in Figure~\ref{fig:almosEquitable4x4}(i)
are described as the periodic extensions of a $2\times2$ corner but
the almost equitable partitions in (ii) and (iii) are not in the families
described in \cite{GillisGolubitsky}. As mentioned after Conjecture~\ref{conj:PmPn},
the algorithms in \cite{Aguiar&Dias,KameiLattice} might allow a computation
of the lattice of almost equitable partitions of $P_{m}\oblong P_{n}$
since the eigenvectors of $L$ are known in closed form.
\end{example}

\subsection{Exo-balanced partitions of coupled cell networks}

The generalization of almost equitable partitions from graphs to coupled
cell networks is immediate.
\begin{defn}
Consider a coupled cell network with cell type partition $\mathcal{T}$.
A cell partition $\mathcal{A}$ is called \emph{exo-balanced} if $\mathcal{A}\le\mathcal{T}$
and $d_{B}^{-}(i)=d_{B}^{-}(j)$ for all $i,j\in A\in\mathcal{A}$
and $A\ne B\in\mathcal{A}$ in every monochrome subgraph.
\end{defn}

If there is only one cell type, so $\mathcal{T}$ is the singleton
partition, then $\mathcal{T}$ is exo-balanced. Furthermore, every
balanced partition is exo-balanced, but the converse is not true.
\begin{example}
\begin{figure}
\begin{tabular}{c}
\includegraphics{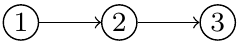}\tabularnewline
\end{tabular}%
\begin{tabular}{ccc}
$123$ & $12|3$ & $1|2|3$\tabularnewline
$\mathcal{A}_{1}$ & $\mathcal{A}_{2}$ & $\boxed{\mathcal{A}_{3}}$\tabularnewline
\end{tabular}

\caption{\label{fig:forPath}A coupled cell network and its exo-balanced partitions.
The boxed partition is balanced.}
\end{figure}

Figure~\ref{fig:forPath} shows a coupled cell network and its exo-balanced
partitions. Only partition $\mathcal{A}_{3}$ is balanced.
\end{example}

The \emph{Laplacian matrix} of a digraph with adjacency matrix $A$
is $L=D-A$, where $D$ is the diagonal matrix with $D_{i,i}=\sum_{j}A_{i,j}$.
The set of Laplacian matrices for a coupled cell network consists
of all the Laplacian matrices of the monochrome subgraphs.
\begin{example}
The Laplacian matrix for the coupled cell network in Figure~\ref{fig:forPath}
is $\left[\begin{smallmatrix}0 & 0 & 0\\
-1 & 1 & 0\\
0 & -1 & 1
\end{smallmatrix}\right].$
\end{example}

The following proposition is a straightforward consequence of the
definitions and the fact that the Laplacian matrix for each arrow
type satisfies $\sum_{j}A_{i,j}=0$. See \cite{AguiarDiasWeighted,CardosoDelormeRama,PecoraClusterSynch,SchaubClusterSynch}.
\begin{prop}
\label{prop:Linv}The set of exo-balanced partitions of a coupled
cell network with set of Laplacian matrices $\mathcal{M}$ and cell
type partition $\mathcal{T}$ is $\Pi_{\mathcal{M}}\cap\downarrow\mathcal{T}$.
\end{prop}

The main application\textcolor{blue}{{} }of exo-balanced partitions
to coupled cell networks involves those cases where two cells do not
interact when they are in the same state. We do not know of a general
theorem along the lines of Proposition~\ref{prop:mainReason}. The
results we know require some specific functional form of the coupling.
We give two such cases. Following the discussion in Subsection~\ref{subsec:balanced-coupled-cell-networks},
we assume without loss of generality that the coupled cell network
has one cell type.

In~\cite{NSS6}, a \emph{difference-coupled vector field} is defined
for coupled cell networks with symmetric adjacency matrix $A$, which
are called graph networks in that paper. Difference coupled vector
fields have the form
\[
\dot{x}_{i}=g(x_{i})+\sum_{j=1}^{n}A_{ij}h(x_{j}-x_{i}),
\]
where $g,h:\R^{k}\rightarrow\R^{k}$. If the coupling function satisfies
$h(0)=0$, then there is no coupling between cells in the same state.
The result \cite[Theorem 3.13]{NSS6} can be generalized to coupled
cell networks with more than one arrow type.
\begin{prop}
\label{prop:nss6exoDigraphs}\textcolor{blue}{{} }Consider a coupled
cell network with a single cell type and a set of adjacency matrices
$\mathcal{M}$. A cell partition $\mathcal{A}$ is exo-balanced if
and only if $\Delta_{\mathcal{A}}^{P}$ is invariant for all coupled
cell network systems of the form\textup{
\[
\dot{x}_{i}=g(x_{i})+\sum_{A\in\mathcal{M}}\sum_{j=1}^{n}A_{i,j}h_{A}(x_{j}-x_{i}),
\]
}with smooth functions \textup{$g,h_{A}:\R^{k}\rightarrow\R^{k}$
satisfying $h_{A}(0)=0$ for all $A\in\mathcal{M}$.}
\end{prop}

\begin{proof}
Assume $\mathcal{A}$ is exo-balanced. It follows easily that $\Delta_{\mathcal{A}}^{P}$
is invariant for the ODE. Conversely, assume $\Delta_{\mathcal{A}}^{P}$
is invariant for all such ODE systems. Thus $\Delta_{\mathcal{A}}^{P}$
is invariant for each of the systems, parameterized by $\tilde{A}\in\mathcal{M}$,
with $g$ the zero function and
\[
h_{A}(x):=\begin{cases}
0, & A\neq\tilde{A}\\
x, & A=\tilde{A}.
\end{cases}
\]
The invariance of $\Delta_{\mathcal{A}}^{P}$ implies that the exo-balanced
condition holds for the monochrome subgraph with adjacency matrix
$\tilde{A}$ as in \cite[Theorem 3.13]{NSS6}.
\end{proof}
Another class of systems considered in \cite{PecoraClusterSynch,SchaubClusterSynch}
has Lagrangian coupling. They considered only one arrow type with
a symmetric adjacency matrix, but the generalization to multiple arrow
types is straightforward.
\begin{prop}
Consider a coupled cell network with a single cell type and a set
of Laplacian matrices $\mathcal{M}$. A \textup{cell partition} $\mathcal{A}$
is exo-balanced if and only if ${\normalcolor \Delta_{\mathcal{A}}^{P}}$
is invariant for all coupled cell network systems of the form
\[
\dot{x}_{i}=g(x_{i})+\sum_{L\in\mathcal{M}}\sum_{j=1}^{n}L{}_{i,j}h_{L}(x_{j}),
\]
with smooth functions $g,h_{L}:\R^{k}\rightarrow\R^{k}$.
\end{prop}

\begin{proof}
The result follows from Proposition~\ref{prop:Linv} and $\sum_{j}L_{i,j}=0$
for all $L\in\mathcal{M}$. We leave the details to the reader.
\end{proof}
Additional applications can be found in \cite{BonaccorsiOttavianoMugnoloPellegrini,Gambuzza}.

\subsection{Weighted cell network systems}

If $M$ is an arbitrary matrix, then we might not have a coupled cell
network whose adjacency matrix is $M$. Such a network is only available
if the matrix has non-negative integer entries. A weighted cell network
\cite{AguiarDiasWeighted,AguiarDiasFerreira} provides a model of
$M$-invariant synchrony subspaces without any restrictions on the
matrix $M$.

Formally, a \emph{weighted coupled cell network} is a coupled cell
network with only one vertex type and one arrow type, together with
a real number weight assigned to every arrow. The \emph{weighted adjacency
matrix} $W$ of the network is defined such that $W_{i,j}$ is the
weight on the arrow from cell $j$ to $i$. This weight is zero if
there is no arrow between these cells. Multiple arrows are not needed
because two arrows from cell $j$ to cell $i$ can be replaced by
a single arrow whose weight is the sum of those two weights. With
this definition, every square matrix is a weighted adjacency matrix
for some weighted cell network. One can define weighted cell networks
with multiple arrow types in the natural way, but we will not pursue
this complication here.

The \emph{in-degree} of a cell $i$ of a weighted cell network is
$d^{-}(i):=\sum_{j=1}^{n}W_{i,j}$.
\begin{defn}
\cite[Definition 2.2]{AguiarDiasFerreira} Consider a weighted coupled
cell network with weighted adjacency matrix $W$. A cell partition
$\mathcal{A}$ is \emph{balanced} if $d^{-}(i)=d^{-}(j)$ for all
$i,j\in A\in\mathcal{A}$ and $B\in\mathcal{A}$.
\end{defn}

The following is immediate from the definitions.
\begin{prop}
A partition $\mathcal{A}$ of a weighted coupled cell network is balanced
if and only if $\mathcal{A}$ is $W$-invariant.
\end{prop}

An \emph{admissible vector field with additive input structure }for
a weighted coupled cell network \cite{AguiarDiasWeighted} is of the
form
\[
\dot{x}_{i}=f(x_{i})+\sum_{j=1}^{n}W_{i,j}g(x_{i},x_{j}),
\]
where $f:\R^{k}\rightarrow\R^{k}$ and $g:\R^{k}\times\R^{k}\rightarrow\R^{k}$
are smooth.

The following result is the main reason why balanced partitions are
important for weighted cell networks.
\begin{prop}
\cite[Theorem 2.4]{AguiarDiasFerreira} \cite[Theorem 2.3]{AguiarDiasWeighted}
A cell partition of a weighted coupled cell network is polysynchronous
for all admissible vector fields with additive input structure and
a given phase space if and only if the partition is balanced.
\end{prop}

The following is defined in \cite[Remark 2.2]{AguiarDiasWeighted}
under the name external equitable.
\begin{defn}
Consider a weighted coupled cell network with weighted adjacency matrix
$W$. A cell partition $\mathcal{A}$ is \emph{exo-balanced} if $d^{-}(i)=d^{-}(j)$
for all $i,j\in A\in\mathcal{A}$ and $A\ne B\in\mathcal{A}$.
\end{defn}

The \emph{Laplacian matrix} of a weighted cell network is $L:=D-W$,
where $D$ is the diagonal matrix satisfying $D_{i,i}=d^{-}(i)$ .
The following is true in this setup as well.
\begin{prop}
A partition $\mathcal{A}$ of a weighted digraph is exo-balanced if
and only if $\mathcal{A}$ is $L$-invariant.
\end{prop}

\begin{defn}
\cite[Remark 2.2]{AguiarDiasWeighted} Given a weighted cell network
$G$ with weighted adjacency matrix $W_{G}$, let $G_{L}$ to be the
weighted cell network with weights $W_{G_{L}}$ satisfying
\[
(W_{G_{L}})_{i,j}=\begin{cases}
(W_{G})_{i,j}, & i\ne j\\
d^{-}(i)-\sum_{j=1}^{n}(W_{G})_{i,j}, & i=j.
\end{cases}
\]
\end{defn}

Note that the Laplacian matrix of $G$ is the weighted adjacency matrix
of $G_{L}$. The following is an immediate consequence.
\begin{prop}
\label{prop:GL}The exo-balanced partitions of a weighted cell network
$G$ are the balanced partitions of $G_{L}$.
\end{prop}

\begin{example}
\begin{figure}
\begin{tabular}{cccccc}
\begin{tabular}{c}
\includegraphics{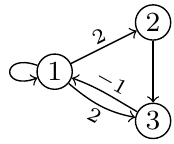}\tabularnewline
\end{tabular} & %
\begin{tabular}{c}
\includegraphics{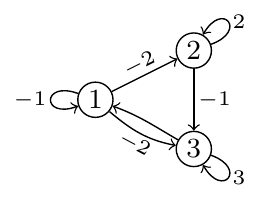}\tabularnewline
\end{tabular} & $\quad$ & %
\begin{tabular}{c}
$123$\tabularnewline
\end{tabular} & %
\begin{tabular}{c}
$1|23$\tabularnewline
\end{tabular} & %
\begin{tabular}{c}
$1|2|3$\tabularnewline
\end{tabular}\tabularnewline
$G$ & $G_{L}$ &  & $\mathcal{A}_{1}$ & $\mathcal{A}_{2}$ & $\mathcal{A}_{3}$\tabularnewline
(i) & (ii) &  & \multicolumn{3}{c}{(iii)}\tabularnewline
\end{tabular}

\caption{\label{fig:weighted}(i) Weighted cell network $G$. Weights of 1
are not printed. (ii) Weighted cell network $G_{L}$. (iii) Balanced
partitions of $G_{L}$ and exo-balanced partitions of $G$.}
\end{figure}
Figure~\ref{fig:weighted}(i) shows a weighted cell network $G$.
Its weighted adjacency, degree, and Laplacian matrices are 
\[
W_{G}=\left[\begin{smallmatrix}1 & 0 & -1\\
2 & 0 & 0\\
2 & 1 & 0
\end{smallmatrix}\right],\qquad D_{G}=\left[\begin{smallmatrix}0 & 0 & 0\\
0 & 2 & 0\\
0 & 0 & 3
\end{smallmatrix}\right],\qquad L_{G}=D_{G}-W_{G}=\left[\begin{smallmatrix}-1 & 0 & 1\\
-2 & 2 & 0\\
-2 & -1 & 3
\end{smallmatrix}\right].
\]
The weighted cell network $G_{L}$ whose adjacency matrix is $W_{G_{L}}=L_{G}$
is shown in Figure~\ref{fig:weighted}(ii). The balanced partitions
of $G_{L}$ are exactly the exo-balanced partitions of $G$. They
are shown in Figure~\ref{fig:weighted}(ii).
\end{example}

The following is an immediate consequence of the formula $L=D-W$
and Proposition~\ref{prop:GL}.
\begin{prop}
If the weighted cell network is regular, that is the in-degree of
every cell is the same, then the exo-balanced partitions are the same
as the balanced partitions.
\end{prop}

\subsection{Finite difference method for PDE}

The current study of $M$-invariant partitions was motivated by work
where finite difference approximations to PDEs gave invariant subspaces
beyond those forced by symmetry \cite{NSS,NSS2,NSS3,NSS5}. Invariant
subspaces are crucial to analyzing bifurcations and also reduce the
dimension of the computations. The region of the PDE is approximated
by a grid of points, and a graph is obtained by joining nearest neighbors
with an edge. The matrix that approximates the Laplacian operator
with boundary conditions on the region is used in the finite difference
method. The graph Laplacian matrix gives the approximation for the
Laplacian operator with zero Neumann boundary conditions \cite{NSS},
so the synchrony subspaces of exo-balanced partitions are invariant.
For other boundary conditions, the subspaces that are invariant under
the matrix approximating the Laplacian include the synchrony subspaces
of balanced partitions.

\subsection{\label{subsec:controllability}Network controllability}

Network controllability concerns the possibility to transfer the state
of a dynamical system with external controls from any given initial
state to any final desired state in finite time \cite{AguilarBahman,Liu2,Liu,ZhangCamlibelCao}.
Such a network system is of the form
\[
\dot{x}=-Mx+Bu,
\]
where the matrix $M$ describes the coupling between the components
of the network, often the adjacency or Laplacian matrix, the matrix
$B$ encodes the application of the external controls $u$, and $x$
encodes the state of the network.

Invariant subspaces of the state space are an obstruction to controllability,
since a target state in an invariant subspace cannot be reached in
finite time  from an initial state not in that invariant subspace.
For example, the invariant subspaces forced by symmetry are an obstacle
toward the controllability of networks if the matrix $B$ is chosen
in a way that preserves the invariant subspace. Similarly, an initial
state in an invariant subspace cannot reach a target state that is
not in the subspace.

It has been shown that symmetries are not necessary for uncontrollability
\cite{rahmani}. Building on this work, \cite[Theorem 3]{AguilarBahman}
showed that if $M$ is the Laplacian of the network and $B$ is chosen
to respect a nontrivial almost equitable partition, then the network
is uncontrollable. They observed that almost equitable partitions
appear frequently in many real-world networks and account for many
of the uncontrollable leader selections that do not come from symmetry.
A result of \cite{ZhangMingKanat} provides an upper bound for the
controllable subspace of networks in terms of maximal almost equitable
partitions. They give an algorithm for finding such partitions given
a set of singleton leaders. Our cir algorithm is a generalization
of their algorithm with the addition that we allow invariance with
respect to multiple, arbitrary matrices.

Further background and references for the subject of network controllability,
particularly as pertains to synchrony and almost equitable partitions,
can be found in \cite{AguilarBahman1,egerstedt,egerstedt2,SchaubClusterSynch}.

\section{Induced partitions}

In this section we develop some preliminary tools. We use the notation
$\left[\!\begin{array}{c|c}
Q & R\end{array}\!\right]$ for the augmented matrix built from the matrices $Q$ and $R$. Given
a matrix $Q\in\mathbb{R}^{n\times k}$, \cite{ZhangMingKanat} constructs
a partition $\psi(Q)$ in $\Pi(n)$ where $i$ and $j$ are in the
same class exactly when the $i$-th and $j$-th rows of $Q$ are equal.
Following \cite[Section 7]{Godsil}, we call $\psi(Q)$ the \emph{partition
induced} by the rows of $Q$.

The following are simple properties of induced partitions.
\begin{prop}
\label{prop:psiProps}The following hold for all $\mathcal{A}\in\Pi(n)$
and $Q,R\in\mathbb{R}^{n\times k}$.
\begin{enumerate}
\item $\psi(P(\mathcal{A}))=\mathcal{A}$;
\item $\psi(\left[\!\begin{array}{c|c}
Q & R\end{array}\!\right])=\psi(Q)\wedge\psi(R)$;
\item $\Col(Q)\subseteq\Col(R)$ implies $\psi(Q)\ge\psi(R)$.
\end{enumerate}
\end{prop}

\begin{proof}
Statements (1) and (2) are clear from the definitions. To show Statement~(3)
assume $\Col(Q)\subseteq\Col(R)$. Then every column of $Q$ is a
linear combination of the columns of $R$. So equal rows in $R$ correspond
to equal rows in $Q$. Hence equivalent elements in $\psi(R)$ are
also equivalent in $\psi(Q)$.
\end{proof}
\begin{lem}
\label{lem:augment}If $\mathcal{A}\le\psi(Q)$ then $\Col(Q)\subseteq\sys(\mathcal{A})$.
\end{lem}

\begin{proof}
Proposition~\ref{prop:psiProps}(1) implies that $\psi(P(\mathcal{A}))=\mathcal{A}\le\psi(Q)$.
Hence equal rows in $P(\mathcal{A})$ correspond to equal rows in
$Q$. We show that $\Col(Q)\subseteq\Col(P(\mathcal{A}))$. It suffices
to check that the reduced row echelon form of $\left[\!\begin{array}{c|c}
P(\mathcal{A}) & Q\end{array}\!\right]$ does not have any leading columns in the augmented part. This is
true because Gauss-Jordan reduction of the augmented matrix can be
performed using only row swaps and the replacement of repeated rows
with zero rows. During these steps, every eliminated row in $P(\mathcal{A})$
is also eliminated in $Q$.
\end{proof}
\begin{example}
To demonstrate the proof of the previous result, consider $\mathcal{A}=13|2|4$
and 
\[
P(\mathcal{A})=\left[\begin{smallmatrix}1 & 0 & 0\\
0 & 1 & 0\\
1 & 0 & 0\\
0 & 0 & 1
\end{smallmatrix}\right],\qquad Q=\left[\begin{smallmatrix}5 & 6\\
7 & 8\\
5 & 6\\
7 & 8
\end{smallmatrix}\right],
\]
so that $\mathcal{A}\le13|24=\psi(Q)$. The reduced row echelon form
\[
\left[\begin{smallmatrix}1 & 0 & 0 & \vv & 5 & 6\\
0 & 1 & 0 & \vv & 7 & 8\\
0 & 0 & 1 & \vv & 7 & 8\\
0 & 0 & 0 & \vv & 0 & 0
\end{smallmatrix}\right]
\]
of $\left[P(\mathcal{A})\mid Q\right]$ has leading columns in the
first three columns. So $\Col(Q)\subseteq\Col(P(\mathcal{A}))=\sys(\mathcal{A})$.
\end{example}

\begin{prop}
\label{prop:psiTactical}Let $\mathcal{M}=\{M_{1},\ldots,M_{r}\}\subseteq\mathbb{R}^{m\times n}$,
$\mathcal{A}\in\Pi(m)$, and $\mathcal{B}\in\Pi(n)$. Then $M_{l}\sys(\mathcal{B})\subseteq\sys(\mathcal{A})$
for all $l$ if and only if 
\[
\mathcal{A}\le\psi(\left[\!\begin{array}{c|c|c}
M_{1}P(\mathcal{B}) & \cdots & M_{r}P(\mathcal{B})\end{array}\!\right]).
\]
\end{prop}

\begin{proof}
First assume that $M_{l}\sys(\mathcal{B})\subseteq\sys(\mathcal{A})$
for all $l$. Then 
\[
\Col(M_{l}P(\mathcal{B}))=M_{l}\Col(P(\mathcal{B}))\subseteq\Col(P(\mathcal{A}))
\]
for all $l$. Proposition~\ref{prop:psiProps}(3) and (1) imply 
\[
\psi(\left[\!\begin{array}{c|c|c}
M_{1}P(\mathcal{B}) & \cdots & M_{r}P(\mathcal{B})\end{array}\!\right])\ge\psi(P(\mathcal{A}))=\mathcal{A}.
\]

Now assume that $\mathcal{A}\le\psi(\left[\!\begin{array}{c|c|c}
M_{1}P(\mathcal{B}) & \cdots & M_{r}P(\mathcal{B})\end{array}\!\right])$. Then $\mathcal{A}\le\psi(M_{l}P(\mathcal{B}))$ by Proposition~\ref{prop:psiProps}(2).
Hence $\Col(M_{l}P(\mathcal{B}))\subseteq\sys(\mathcal{A})$ by Lemma~\ref{lem:augment}.
\end{proof}
The following result is the special case of Proposition~\ref{prop:psiTactical}
when $m=n$ and $\mathcal{B}=\mathcal{A}$.
\begin{cor}
\label{cor:psiM}Let $\mathcal{M}=\{M_{1},\ldots,M_{r}\}\subseteq\mathbb{R}^{n\times n}$.
Then $\mathcal{A}\in\Pi_{\mathcal{M}}$ if and only if 
\[
\mathcal{A}\le\psi(\left[\!\begin{array}{c|c|c}
M_{1}P(\mathcal{A}) & \cdots & M_{r}P(\mathcal{A})\end{array}\!\right]).
\]
\end{cor}

\section{Coarsest invariant refinement}

In this section we develop an algorithm that finds the coarsest invariant
partition that is not larger than a given partition.
\begin{prop}
If $\mathcal{M}\subseteq\mathbb{R}^{n\times n}$ and $\mathcal{A}\in\Pi(n)$,
then $\bigvee(\Pi_{\mathcal{M}}(n)\cap\downarrow\mathcal{A})$ is
in $\Pi_{\mathcal{M}}(n)$.
\end{prop}

\begin{proof}
Note that the down-set $\downarrow\mathcal{A}$ is taken in $\Pi(n)$.
Proposition~\ref{prop:vee} implies that 
\[
\begin{aligned}\sys(\bigvee(\Pi_{\mathcal{M}}(n)\cap\downarrow\mathcal{A})) & =\bigcap\{\sys(\mathcal{B})\mid\mathcal{B}\in\Pi_{\mathcal{M}}(n)\cap\downarrow\mathcal{A}\}\end{aligned}
\]
is the intersection of $\mathcal{M}$-invariant subspaces. Hence $\sys(\bigvee(\Pi_{\mathcal{M}}\cap\downarrow\mathcal{A}))$
is also $\mathcal{M}$-invariant.
\end{proof}
\begin{defn}
We define $\cip_{\mathcal{M}}:\Pi(n)\to\Pi_{\mathcal{M}}(n)$ where
$\cip_{\mathcal{M}}(\mathcal{A}):=\bigvee(\Pi_{\mathcal{M}}(n)\cap\downarrow\mathcal{A})$
is the coarsest $\mathcal{M}$-invariant partition that is not larger
than $\mathcal{A}$. We refer to $\cip_{\mathcal{M}}(\mathcal{A})$
as the \emph{coarsest invariant refinement} of $\mathcal{A}$. 
\end{defn}

The following is a recursive process for finding $\cip_{\mathcal{M}}(\mathcal{A})$
for any $\mathcal{A}$ and for a set $\mathcal{M}$ of matrices. We
refer to it as the $\cip$ algorithm. It is a generalization of the
algorithm in \cite[Theorem 5]{ZhangMingKanat}. Their algorithm finds
$\cip_{M}(\mathcal{A})$ for a single matrix $M$, applied to partitions
$\mathcal{A}$ determined by leader sets.
\begin{prop}
\label{prop:cip}Let $\mathcal{M}=\{M_{1},\ldots,M_{m}\}\subseteq\mathbb{R}^{n\times n}$.
Given a partition $\mathcal{A}_{0}:=\mathcal{A}\in\Pi(n)$, recursively
define $\mathcal{A}_{k+1}:=\psi(\left[\!\begin{array}{c|c}
P(\mathcal{A}_{k}) & Q_{k}\end{array}\!\right])$, where
\[
Q_{k}:=\left[\!\begin{array}{c|c|c}
M_{1}P(\mathcal{A}_{k}) & \cdots & M_{m}P(\mathcal{A}_{k})\end{array}\!\right].
\]
There is a $k_{0}$ such that $\mathcal{A}_{k}=\cip_{\mathcal{M}}(\mathcal{A})$
for all $k\ge k_{0}$.
\end{prop}

\begin{proof}
Proposition~\ref{prop:psiProps}(1) and $(2)$ imply that 
\[
\mathcal{A}_{k+1}=\psi(\left[\!\begin{array}{c|c}
P(\mathcal{A}_{k}) & Q_{k}\end{array}\!\right])=\psi(P(\mathcal{A}_{k}))\wedge\psi(Q_{k})=\mathcal{A}_{k}\wedge\psi(Q_{k}).
\]
Hence $\mathcal{A}_{k+1}\le\mathcal{A}_{k}$ for all $k$.

We show that $\mathcal{A}_{k}\in\Pi_{\mathcal{M}}$ if and only if
$\mathcal{A}_{k+1}=\mathcal{A}_{k}$. If $\mathcal{A}_{k}\in\Pi_{\mathcal{M}}$
then the forward direction of Corollary~\ref{cor:psiM} implies that
$\mathcal{A}_{k}\le\psi(Q_{k})$. Hence $\mathcal{A}_{k}\le\mathcal{A}_{k}\wedge\psi(Q_{k})=\mathcal{A}_{k+1}$,
and so $\mathcal{A}_{k+1}=\mathcal{A}_{k}$. If $\mathcal{A}_{k+1}=\mathcal{A}_{k}$
then
\[
\mathcal{A}_{k}=\mathcal{A}_{k+1}=\mathcal{A}_{k}\wedge\psi(Q_{k})\le\psi(Q_{k}),
\]
and so $\mathcal{A}_{k}\in\Pi_{\mathcal{M}}$ by the backward direction
of Corollary~\ref{cor:psiM}.

Hence $\mathcal{A}_{k}\notin\Pi_{\mathcal{M}}$ implies $\mathcal{A}_{k+1}<\mathcal{A}_{k}$.
Since $\Pi$ is finite and the discrete partition, which is the minimum
element of $\Pi$, belongs to $\Pi_{\mathcal{M}}$, the sequence eventually
produces a partition $\mathcal{A}_{k_{0}}$ in $\Pi_{\mathcal{M}}$.
From this point on the sequence remains constant.

Now we use induction on $k$ to show that $\cip_{\mathcal{M}}(\mathcal{A})\le\mathcal{A}_{k}$
for all $k$. This will imply that $\cip_{\mathcal{M}}(\mathcal{A})=\mathcal{A}_{k_{0}}$.
The base step is clear from the definition. For the inductive step
assume that $\cip_{\mathcal{M}}(\mathcal{A})\le\mathcal{A}_{k}$ for
some $k$. Then $\sys(\cip_{\mathcal{M}}(\mathcal{A}))\supseteq\sys(\mathcal{A}_{k})=\Col(P(\mathcal{A}_{k}))$.
Hence
\[
\sys(\cip_{\mathcal{M}}(\mathcal{A}))\supseteq M_{i}\sys(\cip_{\mathcal{M}}(\mathcal{A}))\supseteq M_{i}\sys(\mathcal{A}_{k})=M_{i}\Col(P(\mathcal{A}_{k}))=\Col(M_{i}P(\mathcal{A}_{k}))
\]
for all $M_{i}\in\mathcal{M}$ by the $M_{i}$-invariance of $\cip_{\mathcal{M}}(\mathcal{A})$.
This implies that $\sys(\cip_{\mathcal{M}}(\mathcal{A}))\supseteq\Col(Q_{k})$.
Hence
\[
\Col(P(\cip_{\mathcal{M}}(\mathcal{A})))=\sys(\cip_{\mathcal{M}}(\mathcal{A}))\supseteq\Col\left[\!\begin{array}{c|c}
P(\mathcal{A}_{k}) & Q_{k}\end{array}\!\right].
\]
Now Proposition~\ref{prop:psiProps}(3) gives
\[
\cip_{\mathcal{M}}(\mathcal{A})=\psi(P(\cip_{\mathcal{M}}(\mathcal{A})))\le\psi(\left[\!\begin{array}{c|c}
P(\mathcal{A}_{k}) & Q_{k}\end{array}\!\right])=\mathcal{A}_{k+1}.
\]
\end{proof}
\begin{note}
\begin{figure}
\begin{centering}
\fbox{\begin{minipage}[c][1\totalheight][t]{0.7\columnwidth}%
\texttt{\small{}input: M = matrix, c = coloring vector}{\small\par}

\texttt{\small{}output: MP = M times the characteristic matrix}{\small\par}

\texttt{\small{}MP = zero matrix}{\small\par}

\texttt{\small{}for i in $\{1,\ldots,n\}$}{\small\par}

\texttt{\small{}$\qquad$for j in $\{1,\ldots,n\}$}{\small\par}

\texttt{\small{}$\qquad$$\qquad$MP{[}i{]}{[}c{[}j{]}{]} += M{[}i{]}{[}j{]}}{\small\par}%
\end{minipage}}
\par\end{centering}
\begin{centering}
\par\end{centering}
\caption{\label{fig:pseudo}Efficient computation of $MP(\mathcal{A})$ using
$c(\mathcal{A})$ avoiding the computation of $P(\mathcal{A})$.}
\end{figure}

The cir algorithm can be implemented using only the coloring vector,
without ever computing the characteristic matrix. This makes the implementation
perform faster. 

In the recursive definition, $\mathcal{A}_{k+1}:=\psi(\left[\!\begin{array}{c|c}
P(\mathcal{A}_{k}) & Q_{k}\end{array}\!\right])$ can be replaced by $\mathcal{A}_{k+1}:=\psi(\left[\!\begin{array}{c|c}
c(\mathcal{A}_{k}) & Q_{k}\end{array}\!\right])$ since two rows of $P(\mathcal{A}_{k})$ are equal if and only if
the corresponding rows of $c(\mathcal{A}_{k})$ are equal.

The matrix product $MP(\mathcal{A})$ can also be computed faster
without full matrix multiplication, as shown in the pseudo code segment
of Figure~\ref{fig:pseudo}. To explain the code segment, let $(c_{1},\ldots,c_{n}):=c(\mathcal{A})$,
so that $P(\mathcal{A})_{j,a}=\delta_{c_{j},a}$. Then $(MP(\mathcal{A}))_{i,a}=\sum_{j}M_{i,j}\delta_{c_{j},a}$.
So the code adds $M_{i,j}$ to $(MP(\mathcal{A}))_{i,a}$ whenever
$a=c_{j}$.

If the matrices in $\mathcal{M}$ are sparse, then a further speedup
is possible with a modification of the code segment that uses sparse
matrix representations.
\end{note}

\begin{example}
\label{exa:cip}
\begin{figure}
\begin{tabular}{cccccc}
$k$ & $\mathcal{A}_{k}$ & $c(\mathcal{A}_{k})$ & $P(\mathcal{A}_{k})$ & $\left[\!\begin{array}{c|c}
c(\mathcal{A}_{k}) & Q_{k}\end{array}\!\right]$ & vertex coloring\tabularnewline
\hline 
 &  &  &  &  & \tabularnewline
$0$ & ${\scriptstyle 14|235}$ & ${\scriptstyle (1,2,2,1,2)}$ & $\left[\begin{smallmatrix}1 & 0\\
0 & 1\\
0 & 1\\
1 & 0\\
0 & 1
\end{smallmatrix}\right]$ & $\left[\begin{smallmatrix}1 & \vv & 0 & 0 & \vv & 0 & 1\\
2 & \vv & 0 & 1 & \vv & 1 & 0\\
2 & \vv & 0 & 0 & \vv & 1 & 0\\
1 & \vv & 0 & 0 & \vv & 0 & 1\\
2 & \vv & 0 & 0 & \vv & 1 & 0
\end{smallmatrix}\right]$ & \includegraphics{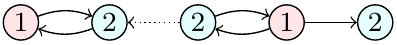}\tabularnewline
 &  &  &  &  & \tabularnewline
$1$ & ${\scriptstyle 14|2|35}$ & ${\scriptstyle (1,2,3,1,3)}$ & $\left[\begin{smallmatrix}1 & 0 & 0\\
0 & 1 & 0\\
0 & 0 & 1\\
1 & 0 & 0\\
0 & 0 & 1
\end{smallmatrix}\right]$ & $\left[\begin{smallmatrix}1 & \vv & 0 & 0 & 0 & \vv & 0 & 1 & 0\\
2 & \vv & 0 & 0 & 1 & \vv & 1 & 0 & 0\\
3 & \vv & 0 & 0 & 0 & \vv & 1 & 0 & 0\\
1 & \vv & 0 & 0 & 0 & \vv & 0 & 0 & 1\\
3 & \vv & 0 & 0 & 0 & \vv & 1 & 0 & 0
\end{smallmatrix}\right]$ & \includegraphics{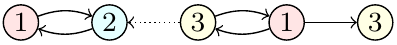}\tabularnewline
 &  &  &  &  & \tabularnewline
$2$ & ${\scriptstyle 1|2|35|4}$ & ${\scriptstyle (1,2,3,4,3)}$ & $\left[\begin{smallmatrix}1 & 0 & 0 & 0\\
0 & 1 & 0 & 0\\
0 & 0 & 1 & 0\\
0 & 0 & 0 & 1\\
0 & 0 & 1 & 0
\end{smallmatrix}\right]$ & $\left[\begin{smallmatrix}1 & \vv & 0 & 0 & 0 & 0 & \vv & 0 & 1 & 0 & 0\\
2 & \vv & 0 & 0 & 1 & 0 & \vv & 1 & 0 & 0 & 0\\
3 & \vv & 0 & 0 & 0 & 0 & \vv & 0 & 0 & 0 & 1\\
4 & \vv & 0 & 0 & 0 & 0 & \vv & 0 & 0 & 1 & 0\\
3 & \vv & 0 & 0 & 0 & 0 & \vv & 0 & 0 & 0 & 1
\end{smallmatrix}\right]$ & \includegraphics{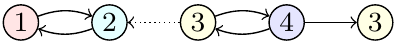}\tabularnewline
\end{tabular}

\caption{\label{fig:cipSteps}The steps of the cir algorithm, starting with
$\mathcal{A}_{0}=14|235$. The third and fifth rows of $[c(\mathcal{A}_{2})\:|\:Q_{2}]$
are equal, so $\mathcal{A}_{2}=\mathcal{A}_{3}=\protect\cip(\mathcal{A}_{0})$.}
\end{figure}

Figure~\ref{fig:cipSteps} shows the steps for finding $\cip_{\mathcal{M}}(14|235)=1|2|35|4$
for a cell network with adjacency matrices
\[
\mathcal{M}=\{\left[\begin{smallmatrix}0 & 0 & 0 & 0 & 0\\
0 & 0 & 1 & 0 & 0\\
0 & 0 & 0 & 0 & 0\\
0 & 0 & 0 & 0 & 0\\
0 & 0 & 0 & 0 & 0
\end{smallmatrix}\right],\left[\begin{smallmatrix}0 & 1 & 0 & 0 & 0\\
1 & 0 & 0 & 0 & 0\\
0 & 0 & 0 & 1 & 0\\
0 & 0 & 1 & 0 & 0\\
0 & 0 & 0 & 1 & 0
\end{smallmatrix}\right]\}.
\]
Since $\mathcal{M}$ contains adjacency matrices, the matrix multiplication
in the cir algorithm can be interpreted as counting incoming arrows.
The entries $(M_{1}P(\mathcal{A}))_{i,a}$ and $(M_{2}P(\mathcal{A}))_{i,a}$
of the matrices in $Q_{k}$ are the number of dashed and solid arrows,
respectively, coming into vertex $i$ from the set of cells with color
$a$. Every step of the cir algorithm breaks up a partition according
to the isomorphism classes of the input sets.

At the first step of the algorithm, the cells with color 2 are broken
into two colors, since vertices 3 and 5 each receive one solid arrow
from color 1, but vertex 2 receives a dashed arrow from color 2 and
a solid arrow from color 1. At the second step of the algorithm, the
two cells with color 1 are separated into two colors because the solid
incoming arrows come from different colors.
\end{example}

\section{The split and cir algorithm for finding the lattice of invariant
partitions}

Proposition~\ref{prop:spliting} and the cir algorithm can be combined
into a split and cir algorithm to find all $\mathcal{M}$-invariant
partitions reasonably quickly. We start with the singleton partition
$\mathcal{A}=\{1,\ldots,n\}$. The cir algorithm produces $\cip_{\mathcal{M}}(\mathcal{A})$
which is the coarsest $\mathcal{M}$-invariant partition. We put this
partition into a queue of $\mathcal{M}$-invariant partitions to further
analyze. For each $\mathcal{M}$-invariant partition $\mathcal{A}$
in the queue, we find each lower cover $\mathcal{B}$ of $\mathcal{A}$
by splitting one of the classes of $\mathcal{A}$. Then we use the
cir algorithm to find $\cip_{\mathcal{M}}(\mathcal{B})$ and add it
to the queue. The algorithm stops when the queue is empty. Figure~\ref{fig:splitCip}
shows the pseudo code for the split and cir algorithm.

\begin{figure}
\begin{centering}
\fbox{\begin{minipage}[c][1\totalheight][t]{0.7\columnwidth}%
\texttt{\small{}input: list of matrices}{\small\par}

\texttt{\small{}output: invPartitions $=$ set of invariant partitions}{\small\par}

\texttt{\small{}$\mathcal{A}=$ singleton partition}{\small\par}

\texttt{\small{}queue.push$(\cip(\mathcal{A}))$}{\small\par}

\texttt{\small{}while queue not empty}{\small\par}

\texttt{\small{}$\qquad$$\mathcal{A}=$ queue.pop()}{\small\par}

\texttt{\small{}$\qquad$find lower covers of $\mathcal{A}$ by splitting
a class of $\mathcal{A}$}{\small\par}

\texttt{\small{}$\qquad$for each $\mathcal{B}$ lower cover of $\mathcal{A}$}{\small\par}

\texttt{\small{}$\qquad$$\qquad$$\mathcal{B}=\cip(\mathcal{B})$}{\small\par}

\texttt{\small{}$\qquad$$\qquad$invPartitions.add$(\mathcal{B})$}{\small\par}

\texttt{\small{}$\qquad$$\qquad$if $\mathcal{B}$ is not in queue}{\small\par}

\texttt{\small{}$\qquad$$\qquad$$\qquad$queue.push$(\mathcal{B})$}{\small\par}%
\end{minipage}}
\par\end{centering}
\caption{\label{fig:splitCip}The split and cir algortihm for finding the $\mathcal{M}$-invariant
partitions.}
\end{figure}

The algorithm can be further improved by considering symmetries of
$\mathcal{M}$. We implemented this algorithm in C++. We compiled
the code with the gnu compiler gcc and ran it on a 3.4 GHz Intel i7-3770
CPU. We also have a Sage cell \cite{sagemath} available on the companion
web page \cite{invariantWEB}.
\begin{example}
Figure~\ref{fig:posetAlgo} shows the steps of our algorithm to find
the 4 balanced partitions that are invariant under the in-adjacency
matrix $A$ of the digraph shown. The number of partitions in $\Pi(7)$
is the Bell number $B_{7}=877$. Our algorithm visits only $10$ of
these partitions to find all invariant partitions.
\end{example}

\begin{figure}
\begin{tabular}{cc}
\begin{tabular}{c}
\includegraphics{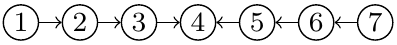}\tabularnewline
\tabularnewline
$A=\left[\begin{smallmatrix}0 & 0 & 0 & 0 & 0 & 0 & 0\\
1 & 0 & 0 & 0 & 0 & 0 & 0\\
0 & 1 & 0 & 0 & 0 & 0 & 0\\
0 & 0 & 1 & 0 & 1 & 0 & 0\\
0 & 0 & 0 & 0 & 0 & 1 & 0\\
0 & 0 & 0 & 0 & 0 & 0 & 1\\
0 & 0 & 0 & 0 & 0 & 0 & 0
\end{smallmatrix}\right]$\tabularnewline
\end{tabular} & %
\begin{tabular}{c}
\includegraphics[clip]{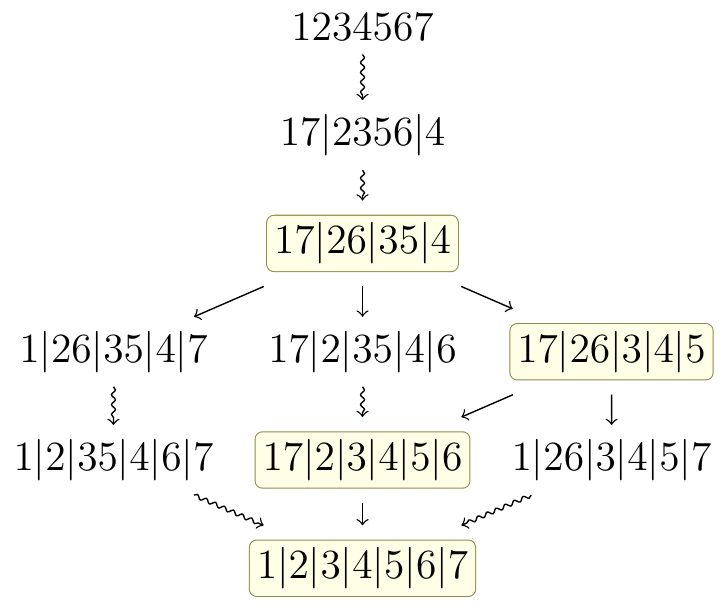}\tabularnewline
\end{tabular}\tabularnewline
(i) & (ii)\tabularnewline
\end{tabular}

\caption{\label{fig:posetAlgo}(i) A cell network and its in-adjacency matrix.
(ii) Steps in the split and cir algorithm for finding the lattice
of balanced partitions. Straight arrows point to lower covers, corresponding
to splits. Squiggly arrows show the steps from $\mathcal{A}_{k}$
to $\mathcal{A}_{k+1}$ in the cir algorithm described in Proposition~\ref{prop:cip}\textcolor{blue}{.}
The balanced partitions are boxed.}
\end{figure}

\begin{example}
Our code finds 37 almost-equitable partitions in 21 orbit classes
of the 27-vertex Sierpinski pre-gasket in less than 6 minutes. The
results are available on the companion web page \cite{invariantWEB}.
\end{example}

\begin{example}
The table below shows the number of balanced partitions of the cycle
graph $C_{n}$ and the running time of our C++ code to find them.

\vspace{1mm}\centerline{%
\begin{tabular}{|c|c|c|c|c|c|c|c|c|c|c|c|}
\hline 
$n$ & 20 & 21 & 22 & 23 & 24 & 25 & 26 & 27 & 28 & 29 & 30\tabularnewline
\hline 
$|\Pi_{A}(n)|$ & 45 & 35 & 37 & 25 & 65 & 33 & 43 & 43 & 59 & 31 & 77\tabularnewline
\hline 
time (sec) & 1 & 2 & 4 & 8 & 16 & 31 & 67 & 131 & 271 & 559 & 1110\tabularnewline
\hline 
\end{tabular}}\vspace{1mm}\noindent The observed run time appears to be of order
$2^{n}$. This is reasonable since the bulk of the time is spent doing
cir to each of the splittings of $12\cdots n$. There are $2^{n-1}-1$
such splittings, since the number of proper nontrivial subsets of
$\left\{ 1,2,\ldots,n\right\} $ is twice the number of splits. The
algorithm of \cite{KameiLattice} checks all of the partitions in
this example and the run time is roughly of order the Bell number,
which grows faster than exponential.

Every orbit partition arises from a factor of $n$, and it can be
shown that the number of orbit partitions is $\sum_{d|n}f(d)$, where
\[
f(d):=\begin{cases}
1, & d\in\{1,2\}\\
d+1, & d>2.
\end{cases}
\]
This count agrees with the second row of the table; for example, $|\Pi_{A}(25)|=f(1)+f(5)+f(25)=1+6+26=33$.
We conjecture that all balanced partitions are orbit partitions. For
this graph the eigenvectors of $A$ are known in closed form, and
the algorithm in \cite{Aguiar&Dias} could probably be used to prove
this conjecture. Our conjecture is tantalizingly close to \cite[Corollary 4.7]{Antoneli2},
and perhaps the techniques of that paper can be extended to prove
it.
\end{example}

\begin{example}
The table below shows the running times of our C++ code to find the
$10$ balanced partitions of $P_{n}\square P_{n}$. 

\vspace{1mm}\centerline{%
\begin{tabular}{|c|c|c|c|c|c|c|c|c|c|c|c|}
\hline 
$n$ & 20 & 21 & 22 & 23 & 24 & 25 & 26 & 27 & 28 & 29 & 30\tabularnewline
\hline 
time (sec) & 39 & 57 & 89 & 127 & 191 & 267 & 388 & 527 & 757 & 1017 & 1355\tabularnewline
\hline 
\end{tabular}}\vspace{1mm}\noindent The run times seem consistent with a power
law of order $n^{8.8}$. The coarsest invariant partition for this
family is the orbit partition, so the largest element has size 8.
This bound on the size of the largest element of the coarsest partition
avoids the exponential growth in run time observed for the $C_{n}$
family.
\end{example}

The results for these two families of graphs show that it is very
difficult to estimate the complexity of our split and cir algorithm.
Certainly the size of the largest class of the coarsest invariant
partition is important.\textcolor{red}{{} }For the complete graph, every
partition is a fixed point of cir, so the split and cir algorithm
is essentially brute force. If the coarsest invariant partition is
the discrete partition, then our  algorithm does no splitting. Our
split and cir algorithm performs best compared to existing algorithms
if cir gives significant refinement when applied to the splits.

\section{Tactical decompositions}

The following is a generalization of the tactical decompositions studied
in design theory \cite{Betten,CameronLiebler,Dembowski,DembowskiBook},
also called equitable partitions in \cite[Section 12.7]{Godsil}.
A tactical decomposition of an incidence structure is a generalization
of the automorphism group of the incidence structure.
\begin{example}
Let $\Pi(m,n)$ be the product lattice $\Pi(m)\times\Pi(n)$. That
is, $(\mathcal{A},\mathcal{B})\le(\mathcal{C},\mathcal{D})$ if $\mathcal{A}\le\mathcal{C}$
and $\mathcal{D}\le\mathcal{B}$, while the lattice operations $\vee$
and $\wedge$ are defined coordinatewise. If $(\mathcal{A},\mathcal{B})\le(\mathcal{C},\mathcal{D})$
then we say that $(\mathcal{A},\mathcal{B})$ is \emph{finer} than
$(\mathcal{C},\mathcal{D})$, or that $(\mathcal{C},\mathcal{D})$
is \emph{coarser} than $(\mathcal{A},\mathcal{B})$. We simply write
$\Pi$ if $m$ and $n$ is clear from the context.
\end{example}

\begin{defn}
Let $\mathcal{M}\subseteq\mathbb{R}^{m\times n}$. An element $(\mathcal{A},\mathcal{B})$
of $\Pi(m,n)$ is a \emph{tactical decomposition} of $\mathcal{M}$
if $M\sys(\mathcal{B})\subseteq\sys(\mathcal{A})$ and $M^{T}\sys(\mathcal{A})\subseteq\sys(\mathcal{B})$
for all $M\in\mathcal{M}$. The set of tactical decompositions is
denoted by $\Pi{}_{\mathcal{M}}(m,n)$.
\end{defn}

Note that if $m=n$ and $\mathcal{A}$ is an $\mathcal{M}$-invariant
partition, then $(\mathcal{A},\mathcal{A})$ is a tactical decomposition.
On the other hand $\Pi_{\mathcal{M}}(n,n)\ne\Pi_{\mathcal{M}}(n)$.

An incidence structure is a triple consisting\textcolor{blue}{{} }of
a set of points, a set of lines, and an incidence \textcolor{black}{relation}
determining which points are incident to which lines. Any graph can
be considered as an incidence structure, where the vertices of the
graph are the points and the edges of the graph are the lines in the
incidence structure.
\begin{example}
\begin{figure}
\setlength{\tabcolsep}{1pt}

\begin{tabular}{cccccc}
\includegraphics{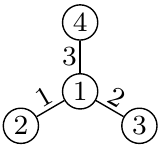} & \includegraphics{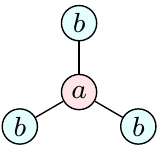} & \includegraphics{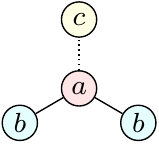} & \includegraphics{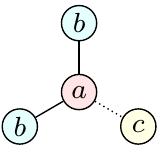} & \includegraphics{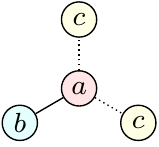} & \includegraphics{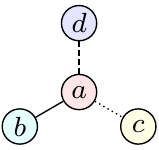}\tabularnewline
\multirow{2}{*}{%
\begin{tabular}{c}
\includegraphics{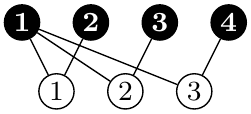}\tabularnewline
\end{tabular}} & %
\begin{tabular}{c}
$1|234$\tabularnewline
$123$\tabularnewline
\end{tabular} & %
\begin{tabular}{c}
$1|23|4$\tabularnewline
$12|3$\tabularnewline
\end{tabular} & %
\begin{tabular}{c}
$1|24|3$\tabularnewline
$13|2$\tabularnewline
\end{tabular} & %
\begin{tabular}{c}
$1|2|34$\tabularnewline
$1|23$\tabularnewline
\end{tabular} & %
\begin{tabular}{c}
$1|2|3|4$\tabularnewline
$1|2|3$\tabularnewline
\end{tabular}\tabularnewline
 & $(\mathcal{A}_{1},\mathcal{B}_{1})$ & $(\mathcal{A}_{2},\mathcal{B}_{2})$ & $(\mathcal{A}_{3},\mathcal{B}_{3})$ & $(\mathcal{A}_{4},\mathcal{B}_{4})$ & $(\mathcal{A}_{5},\mathcal{B}_{5})$\tabularnewline
(i) &  &  & (ii) &  & \tabularnewline
\end{tabular}%
\begin{tabular}{c}
\includegraphics{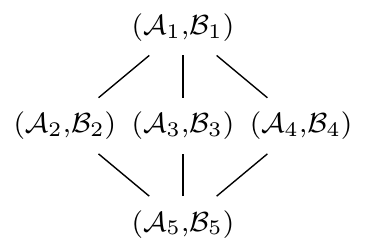}\tabularnewline
\vspace{3.5mm}\tabularnewline
(iii)\tabularnewline
\end{tabular}

\caption{\label{fig:TacticalEx-0}(i) The star graph $K_{1,3}$ and its incidence
graph. (ii) The tactical decompositions of the incidence matrix. (iii)
Lattice of tactical decompositions.}
\end{figure}

Figure~\ref{fig:TacticalEx-0}(i) shows the star graph $K_{1,3}$
and its incidence graph. The incidence graph is a bipartite graph
with black vertices representing points, white vertices representing
lines, and edges representing incidence. The incidence matrix is 
\[
M=\left[\begin{smallmatrix}1 & 1 & 1\\
1 & 0 & 0\\
0 & 1 & 0\\
0 & 0 & 1
\end{smallmatrix}\right].
\]
The tactical decompositions of $M$ are shown in Figure \ref{fig:TacticalEx-0}(ii).
\end{example}

While an incidence structure has only one type of incidence, our definition
of tactical decompositions allows more than one edge type in the incidence
graph as in the next example.
\begin{example}
\begin{figure}
\setlength{\tabcolsep}{1pt}

\begin{tabular}{cccc}
\includegraphics{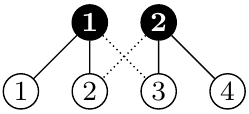} &  & \includegraphics{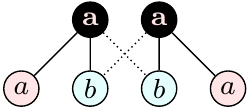} & \includegraphics{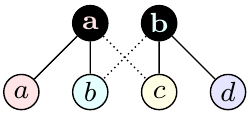}\tabularnewline
$M_{1}=\left[\begin{smallmatrix}1 & 1 & 0 & 0\\
0 & 0 & 1 & 1
\end{smallmatrix}\right]$ & ~ & %
\begin{tabular}{c}
$12$\tabularnewline
$14|23$\tabularnewline
\end{tabular} & %
\begin{tabular}{c}
$1|2$\tabularnewline
$1|2|3|4$\tabularnewline
\end{tabular}\tabularnewline
$M_{2}=\left[\begin{smallmatrix}0 & 0 & 1 & 0\\
0 & 1 & 0 & 0
\end{smallmatrix}\right]$ &  & $(\mathcal{A}_{1},\mathcal{B}_{1})$ & $(\mathcal{A}_{2},\mathcal{B}_{2})$\tabularnewline
(i) &  & \multicolumn{2}{c}{(ii)}\tabularnewline
\end{tabular}~%
\begin{tabular}{c}
\includegraphics{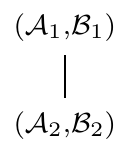}\tabularnewline
\tabularnewline
\tabularnewline
(iii)\tabularnewline
\end{tabular}

\caption{\label{fig:TacticalEx-1}(i) Edge colored incidence graph of an incidence
structure. (ii) Tactical decompositions of $\mathcal{M}=\{M_{1},M_{2}\}$.
(iii) Lattice of tactical partitions.}
\end{figure}

Figure~\ref{fig:TacticalEx-1} shows the tactical decompositions
of $\mathcal{M}=\{M_{1},M_{2}\}$ that contains the incidence matrices
of an edge colored incidence graph of the incidence structure containing
the black points and white lines.
\end{example}

\begin{example}
\begin{figure}
\setlength{\tabcolsep}{1pt}

\begin{tabular}{ccccccc}
\begin{tabular}{c}
\includegraphics{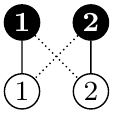}\tabularnewline
$\left[\begin{smallmatrix}1 & 0\\
0 & 1
\end{smallmatrix}\right]$\tabularnewline
\tabularnewline
$\left[\begin{smallmatrix}0 & 1\\
1 & 0
\end{smallmatrix}\right]$\tabularnewline
\tabularnewline
\tabularnewline
\end{tabular} & %
\begin{tabular}{c}
\begin{tabular}{c}
$12$\tabularnewline
$12$\tabularnewline
\end{tabular}\tabularnewline
$(\mathcal{A}_{1},\mathcal{B}_{1})$\tabularnewline
\tabularnewline
\begin{tabular}{c}
$1|2$\tabularnewline
$1|2$\tabularnewline
\end{tabular}\tabularnewline
$(\mathcal{A}_{2},\mathcal{B}_{2})$\tabularnewline
\tabularnewline
\end{tabular} & %
\begin{tabular}{c}
\includegraphics{K24Lattice}\tabularnewline
\tabularnewline
\tabularnewline
\tabularnewline
\tabularnewline
\tabularnewline
\end{tabular} & ~ & %
\begin{tabular}{c}
\includegraphics{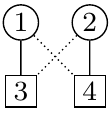}\tabularnewline
$\left[\begin{smallmatrix}0 & 0 & 1 & 0\\
0 & 0 & 0 & 1\\
1 & 0 & 0 & 0\\
0 & 1 & 0 & 0
\end{smallmatrix}\right]$\tabularnewline
\tabularnewline
$\left[\begin{smallmatrix}0 & 0 & 0 & 1\\
0 & 0 & 1 & 0\\
0 & 1 & 0 & 0\\
1 & 0 & 0 & 0
\end{smallmatrix}\right]$\tabularnewline
\end{tabular} & %
\begin{tabular}{ccc}
\begin{tabular}{c}
$1234$\tabularnewline
\end{tabular} &  & %
\begin{tabular}{c}
$14|23$\tabularnewline
\end{tabular}\tabularnewline
$\mathcal{A}_{1}$ &  & $\mathcal{A}_{4}$\tabularnewline
 &  & \tabularnewline
\begin{tabular}{c}
$12|34$\tabularnewline
\end{tabular} &  & %
\begin{tabular}{c}
$1|2|3|4$\tabularnewline
\end{tabular}\tabularnewline
$\mathcal{A}_{2}$ &  & $\mathcal{A}_{5}$\tabularnewline
 &  & \tabularnewline
\begin{tabular}{c}
$13|24$\tabularnewline
\end{tabular} &  & \tabularnewline
$\mathcal{A}_{3}$ &  & \tabularnewline
\end{tabular} & %
\begin{tabular}{c}
\includegraphics{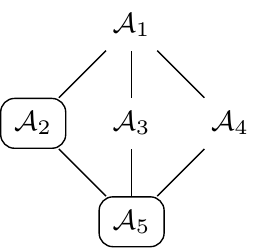}\tabularnewline
\tabularnewline
\tabularnewline
\tabularnewline
\end{tabular}\tabularnewline
(i) & (ii) & (iii) &  & (iv) & (v) & (vi)\tabularnewline
\end{tabular}

\caption{\label{fig:TacticalEx-2}(i) Edge colored incidence graph $K_{2,2}$.
(ii) Tactical decompositions. (iii) Lattice of tactical decompositions.
(iv) Cell network on $K_{2,2}$ with two cell and arrow types. (v)
$\mathcal{M}$-invariant partitions. (vi) Lattice of $\mathcal{M}$-invariant
partitions with balanced partitions circled.}
\end{figure}

Figure~\ref{fig:TacticalEx-2}(ii) shows the tactical decompositions
of $\{\left[\begin{smallmatrix}1 & 0\\
0 & 1
\end{smallmatrix}\right],\left[\begin{smallmatrix}0 & 1\\
1 & 0
\end{smallmatrix}\right]\}$ that contains the incidence matrices of the edge colored incidence
graph $K_{2,2}$ shown in Figure~\ref{fig:TacticalEx-2}(i). Figure~\ref{fig:TacticalEx-2}(v)
shows the $\mathcal{M}$-invariant partitions of the corresponding
cell network shown in Figure~\ref{fig:TacticalEx-2}(iv). Note that
the tactical decompositions correspond to the circled balanced partitions,
which are coarser than the cell type partition $\mathcal{T}=12|34$.
Note the similarity to the coupled cell network of Figure~\ref{fig:balEx2}.

The moral of this example is that tactical decompositions give a compact
way of describing balanced partitions of the incidence graph, considered
as a coupled cell network. The cell type partition in the coupled
cell network has two classes, the points and the lines of the incidence
structure. If the points are listed first and the lines second, then
the adjacency matrix of the incidence graph has the incidence matrix
in the upper right corner. This determines the whole adjacency matrix,
since the lower left corner is the transpose of the incidence matrix,
and there are two zero blocks on the diagonal.
\end{example}

The following is a generalization of Proposition~\ref{prop:supInvariant}.
\begin{prop}
\label{prop:supInvariantTactical}If $S\subseteq\Pi_{\mathcal{M}}(m,n)$
then $\bigvee S\in\Pi_{\mathcal{M}}(m,n)$.
\end{prop}

\begin{proof}
Let $S=\{(\mathcal{A}_{i},\mathcal{B}_{i})\mid i\in I\}$. Then 
\[
\begin{aligned}M\sys(\bigvee_{i\in I}\mathcal{A}_{i}) & =M\bigcap\{\sys(\mathcal{A}_{i})\mid i\in I\}\subseteq\bigcap\{M\sys(\mathcal{A}_{i})\mid i\in I\}\\
 & \subseteq\bigcap\{\sys(\mathcal{B}_{i})\mid i\in I\}=\sys(\bigvee_{i\in I}\mathcal{B}_{i})
\end{aligned}
\]
for all $M\in\mathcal{M}$. Similar computation shows that 
\[
M^{T}\sys(\bigvee_{i\in I}\mathcal{B}_{i})\subseteq\sys(\bigvee_{i\in I}\mathcal{A}_{i})
\]
for all $M\in\mathcal{M}$. Hence $\bigvee S=(\bigvee_{i}\mathcal{A}_{i},\bigvee_{i}\mathcal{B}_{i})\in\Pi_{\mathcal{M}}$.
\end{proof}
\begin{prop}
\label{prop:TacticalLattice}If $\mathcal{M}\subseteq\mathbb{R}^{m\times n}$
then $\Pi_{\mathcal{M}}(m,n)$ is a lattice.
\end{prop}

\begin{proof}
We verify the conditions of Proposition~\ref{prop:latticeCondition}.
The bottom element of $\Pi(m,n)$ is the pair of discrete partitions.
This pair is a tactical decomposition for any $\mathcal{M}$. Let
$S\subseteq\Pi_{\mathcal{M}}(m,n)$ and $\bigvee S$ be the supremum
of $S$ taken in the lattice $\Pi(m,n)$. Then $\bigvee S\in\Pi_{\mathcal{M}}$
by Proposition~\ref{prop:supInvariantTactical}. So the supremum
of $S$ in $\Pi_{\mathcal{M}}$ is $\bigvee S$.
\end{proof}

\section{Coarsest tactical refinement}

In this section we develop an algorithm that finds the coarsest tactical
decomposition that is not larger than a given decomposition.
\begin{prop}
If $\mathcal{M}\subseteq\mathbb{R}^{m\times n}$ and $(\mathcal{A},\mathcal{B})\in\Pi(m,n)$,
then $\bigvee(\Pi_{\mathcal{M}}(m,n)\cap\downarrow(\mathcal{A},\mathcal{B}))$
is in $\Pi_{\mathcal{M}}(m,n)$.
\end{prop}

\begin{proof}
Let $X:=\Pi_{\mathcal{M}}(m,n)\cap\downarrow(\mathcal{A},\mathcal{B})$.
Note that the down-set is taken in $\Pi_{\mathcal{M}}(m,n)$. Define
\[
(\mathcal{A}_{*},\mathcal{B}_{*}):=\bigvee X=(\bigvee\{\mathcal{C}\mid(\mathcal{C},\mathcal{D})\in X\},\bigvee\{\mathcal{D}\mid(\mathcal{C},\mathcal{D})\in X\}).
\]
Proposition~\ref{prop:vee} implies that
\[
\begin{aligned}\sys(\mathcal{A}_{*}) & =\bigcap\{\sys(\mathcal{C})\mid(\mathcal{C},\mathcal{D})\in X\},\\
\sys(\mathcal{B}_{*}) & =\bigcap\{\sys(\mathcal{D})\mid(\mathcal{C},\mathcal{D})\in X\}.
\end{aligned}
\]
Let $M\in\mathcal{M}$. Then $M\sys(\mathcal{D})\subseteq\sys(\mathcal{C})$
for all $(\mathcal{C},\mathcal{D})\in X$. So 
\[
\begin{aligned}M\sys(\mathcal{B}_{*}) & =M\bigcap\{\sys(\mathcal{D})\mid(\mathcal{C},\mathcal{D})\in X\}\\
 & =\bigcap\{M\sys(\mathcal{D})\mid(\mathcal{C},\mathcal{D})\in X\}\\
 & \subseteq\bigcap\{\sys(\mathcal{C})\mid(\mathcal{C},\mathcal{D})\in X\}=\sys(\mathcal{A}_{*}).
\end{aligned}
\]
\end{proof}
\begin{defn}
We define $\cip_{\mathcal{M}}:\Pi(m,n)\to\Pi_{\mathcal{M}}(m,n)$
where 
\[
\cip_{\mathcal{M}}(\mathcal{A},\mathcal{B}):=\bigvee(\Pi_{\mathcal{M}}(m,n)\cap\downarrow(\mathcal{A},\mathcal{B}))
\]
 is the \emph{coarsest tactical refinement} that is not larger than
$(\mathcal{A},\mathcal{B})$.
\end{defn}

The following is a recursive process for finding $\cip_{\mathcal{M}}(\mathcal{A},\mathcal{B})$.
\begin{prop}
Let $\mathcal{M}=\{M_{1},\ldots,M_{r}\}\subseteq\mathbb{R}^{m\times n}$.
Given $(\mathcal{A}_{0},\mathcal{B}_{0})=(\mathcal{A},\mathcal{B})\in\Pi(m,n)$,
recursively define $\mathcal{A}_{k+1}:=\psi(\left[\!\begin{array}{c|c}
P(\mathcal{A}_{k}) & Q_{k}\end{array}\!\right])$, where 
\[
Q_{k}:=\left[\!\begin{array}{c|c|c}
M_{1}P(\mathcal{B}_{k}) & \cdots & M_{r}P(\mathcal{B}_{k})\end{array}\!\right]
\]
and $\mathcal{B}_{k+1}:=\psi(\left[\!\begin{array}{c|c}
P(\mathcal{B}_{k}) & R_{k}\end{array}\!\right])$, where 
\[
R_{k}:=\left[\!\begin{array}{c|c|c}
M_{1}^{T}P(\mathcal{A}_{k}) & \cdots & M_{r}^{T}P(\mathcal{A}_{k})\end{array}\!\right].
\]
There is a $k_{0}$ such that $(\mathcal{A}_{k},\mathcal{B}_{k})=\cip_{\mathcal{M}}(\mathcal{A},\mathcal{B})$
for all $k\ge k_{0}$.
\end{prop}

\begin{proof}
Proposition~\ref{prop:psiProps}(1) and $(2)$ imply that 
\[
\mathcal{A}_{k+1}=\psi(\left[\!\begin{array}{c|c}
P(\mathcal{A}_{k}) & Q_{k}\end{array}\!\right])=\psi(P(\mathcal{A}_{k}))\wedge\psi(Q_{k})=\mathcal{A}_{k}\wedge\psi(Q_{k})
\]
and similarly $\mathcal{B}_{k+1}=\mathcal{B}_{k}\wedge\psi(R_{k})$.
Hence $\mathcal{A}_{k+1}\le\mathcal{A}_{k}$ and $\mathcal{B}_{k+1}\le\mathcal{B}_{k}$
for all $k$.

We show that $(\mathcal{A}_{k},\mathcal{B}_{k})\in\Pi_{\mathcal{M}}(m,n)$
if and only if $(\mathcal{A}_{k+1},\mathcal{B}_{k+1})=(\mathcal{A}_{k},\mathcal{B}_{k})$.
If $(\mathcal{A}_{k},\mathcal{B}_{k})\in\Pi_{\mathcal{M}}(m,n)$ then
the forward direction of Proposition~\ref{prop:psiTactical} implies
that $\mathcal{A}_{k}\le\psi(Q_{k})$ and $\mathcal{B}_{k}\le\psi(R_{k})$.
Hence $\mathcal{A}_{k}\le\mathcal{A}_{k}\wedge\psi(Q_{k})=\mathcal{A}_{k+1}$
and $\mathcal{B}_{k}\le\mathcal{B}_{k}\wedge\psi(R_{k})=\mathcal{B}_{k+1}$.
Thus $\mathcal{A}_{k+1}=\mathcal{A}_{k}$ and $\mathcal{B}_{k+1}=\mathcal{B}_{k}$.
If $(\mathcal{A}_{k+1},\mathcal{B}_{k+1})=(\mathcal{A}_{k},\mathcal{B}_{k})$
then 
\[
\begin{aligned}\mathcal{A}_{k} & =\mathcal{A}_{k+1}=\mathcal{A}_{k}\wedge\psi(Q_{k})\le\psi(Q_{k}),\\
\mathcal{B}_{k} & =\mathcal{B}_{k+1}=\mathcal{B}_{k}\wedge\psi(R_{k})\le\psi(R_{k}),
\end{aligned}
\]
and so $(\mathcal{A}_{k},\mathcal{B}_{k})\in\Pi_{\mathcal{M}}(m,n)$
by the backward direction of Proposition~\ref{prop:psiTactical}.

Hence $(\mathcal{A}_{k},\mathcal{B}_{k})\notin\Pi_{\mathcal{M}}(m,n)$
implies that $\mathcal{A}_{k+1}<\mathcal{A}_{k}$ or $\mathcal{B}_{k+1}<\mathcal{B}_{k}$.
Since $\Pi(m,n)$ is finite and the pair of discrete partitions, which
is the minimum element of $\Pi(m,n)$, belongs to $\Pi_{\mathcal{M}}(m,n)$,
the sequence eventually produces a partition $(\mathcal{A}_{k_{0}},\mathcal{B}_{k_{0}})$
in $\Pi_{\mathcal{M}}(m,n)$. From this point on the sequence remains
constant.

Now we use induction on $k$ to show that $\cip_{\mathcal{M}}(\mathcal{A},\mathcal{B})\le(\mathcal{A}_{k},\mathcal{B}_{k})$
for all $k$. This will imply that $\cip_{\mathcal{M}}(\mathcal{A},\mathcal{B})=(\mathcal{A}_{k_{0}},\mathcal{B}_{k_{0}})$.
The base step is clear from the definition. For the inductive step
let $(\mathcal{C},\mathcal{D}):=\cip_{\mathcal{M}}(\mathcal{A},\mathcal{B})$
and assume that $(\mathcal{C},\mathcal{D})\le(\mathcal{A}_{k},\mathcal{B}_{k})$
for some $k$. Then $\sys(\mathcal{C})\supseteq\sys(\mathcal{A}_{k})=\Col(P(\mathcal{A}_{k}))$
and $\sys(\mathcal{D})\supseteq\sys(\mathcal{B}_{k})=\Col(P(\mathcal{B}_{k}))$.
Hence
\[
\sys(\mathcal{C})\supseteq M_{i}\sys(\mathcal{D})\supseteq M_{i}\Col(P(\mathcal{B}_{k}))=\Col(M_{i}P(\mathcal{B}_{k})),
\]
\[
\sys(\mathcal{D})\supseteq M_{i}^{T}\sys(\mathcal{C})\supseteq M_{i}^{T}\Col(P(\mathcal{A}_{k}))=\Col(M_{i}^{T}P(\mathcal{A}_{k})),
\]
for all $M_{i}\in\mathcal{M}$ by the $M_{i}$-invariance of $(\mathcal{C},\mathcal{D})$.
This implies that $\sys(\mathcal{C})\supseteq\Col(Q_{k})$ and $\sys(\mathcal{D})\supseteq\Col(R_{k})$.
Hence
\[
\Col(P(\mathcal{C}))=\sys(\mathcal{C})\supseteq\Col\left[\!\begin{array}{c|c}
P(\mathcal{A}_{k}) & Q_{k}\end{array}\!\right],
\]
\[
\Col(P(\mathcal{D}))=\sys(\mathcal{D})\supseteq\Col\left[\!\begin{array}{c|c}
P(\mathcal{B}_{k}) & R_{k}\end{array}\!\right].
\]
Now Proposition~\ref{prop:psiProps}(3) gives
\[
\mathcal{C}=\psi(P(\mathcal{C}))\le\psi(\left[\!\begin{array}{c|c}
P(\mathcal{A}_{k}) & Q_{k}\end{array}\!\right])=\mathcal{A}_{k+1},
\]
\[
\mathcal{D}=\psi(P(\mathcal{D}))\le\psi(\left[\!\begin{array}{c|c}
P(\mathcal{B}_{k}) & R_{k}\end{array}\!\right])=\mathcal{B}_{k+1}.
\]
Thus $\cip_{\mathcal{M}}(\mathcal{A},\mathcal{B})=(\mathcal{C},\mathcal{D})\le(\mathcal{A}_{k+1},\mathcal{B}_{k+1})$.
\end{proof}

\section{Finding the lattice of tactical decompositions}

An element $(\mathcal{C},\mathcal{D})$ of $\Pi(m,n)$ covers another
$(\mathcal{A},\mathcal{B})$ if either $\mathcal{A}$ can be constructed
from $\mathcal{C}$ by splitting one of the classes of $\mathcal{C}$
into two nonempty sets, or $\mathcal{B}$ can be constructed from
$\mathcal{D}$ by splitting one of the classes of $\mathcal{D}$ into
two nonempty sets. More precisely, we have the following.
\begin{prop}
\label{prop:splitingTactical}For $(\mathcal{A},\mathcal{B}),(\mathcal{C},\mathcal{D})\in\Pi(m,n)$,
$(\mathcal{A},\mathcal{B})\prec(\mathcal{C},\mathcal{D})$ if and
only if either $\mathcal{A}=\mathcal{C}\cup\{A,B\}\setminus\{C\}$
for some $A,B\subset C\in\mathcal{C}$ such that $C=A\mathring{\cup}B$,
or $\mathcal{B}=\mathcal{D}\cup\{A,B\}\setminus\{C\}$ for some $A,B\subset C\in\mathcal{D}$
such that $C=A\mathring{\cup}B$.
\end{prop}

Proposition~\ref{prop:splitingTactical} and the cir algorithm can
be combined into a split and cir algorithm to find all $\mathcal{M}$-invariant
tactical decompositions. We start with $(\mathcal{A},\mathcal{A})\in\Pi(m,n)$
built from $\mathcal{A}=1\cdots n$. The cir algorithm produces $\cip_{\mathcal{M}}(\mathcal{A},\mathcal{A})$
which is the coarsest $\mathcal{M}$-invariant tactical decomposition.
We put this tactical decomposition into a queue of $\mathcal{M}$-invariant
tactical decompositions to further analyze. For each $\mathcal{M}$-invariant
tactical decomposition $(\mathcal{A},\mathcal{B})$ in the queue,
we find each lower cover $(\mathcal{C},\mathcal{D})$ of $(\mathcal{A},\mathcal{B})$
by splitting one of the classes of $\mathcal{A}$ or one of the classes
of $B$. Then we use the cir algorithm to find $\cip_{\mathcal{M}}(\mathcal{C},\mathcal{D})$
and add it to the queue. The algorithm stops when the queue is empty.
\begin{example}
\begin{figure}
\begin{tabular}{ccccc}
\begin{tabular}{c}
\includegraphics{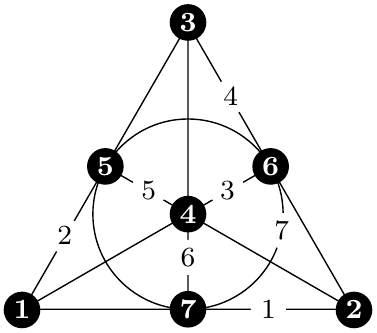}\tabularnewline
\end{tabular} &  & $M=\left[\begin{smallmatrix}1 & 1 & 1 & 0 & 0 & 0 & 0\\
1 & 0 & 0 & 1 & 1 & 0 & 0\\
0 & 1 & 0 & 1 & 0 & 1 & 0\\
0 & 0 & 1 & 0 & 1 & 1 & 0\\
0 & 1 & 0 & 0 & 1 & 0 & 1\\
0 & 0 & 1 & 1 & 0 & 0 & 1\\
1 & 0 & 0 & 0 & 0 & 1 & 1
\end{smallmatrix}\right]$ &  & %
\begin{tabular}{c}
\includegraphics{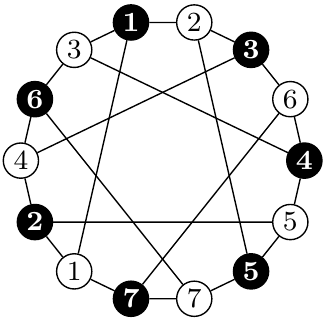}\tabularnewline
\end{tabular}\tabularnewline
(i) &  & (ii) &  & (iii)\tabularnewline
\end{tabular}

\caption{\label{fig:FanoDef}(i) The Fano plane. (ii) The incidence matrix
$M$. (iii) The incidence graph of the Fano plane.}

\end{figure}

\begin{figure}
\setlength{\tabcolsep}{1pt}

\begin{tabular}{cccc}
\begin{tabular}{c}
\includegraphics{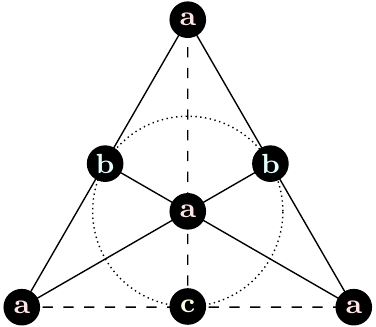}\tabularnewline
\end{tabular} & ~ & %
\begin{tabular}{ccccc}
\begin{tabular}{c}
$1234567$\tabularnewline
$1234567$\tabularnewline
\end{tabular} & ~ & %
\begin{tabular}{c}
$1234|56|7$\tabularnewline
$16|2345|7$\tabularnewline
\end{tabular} & ~ & %
\begin{tabular}{c}
$12|34|56|7$\tabularnewline
$1|2345|6|7$\tabularnewline
\end{tabular}\tabularnewline
$(\mathcal{A}_{1},\mathcal{B}_{1})$ &  & $(\mathcal{A}_{4},\mathcal{B}_{4})$ &  & $(\mathcal{A}_{7},\mathcal{B}_{7})$\tabularnewline
 &  &  &  & \tabularnewline
\begin{tabular}{c}
$123456|7$\tabularnewline
$167|2345$\tabularnewline
\end{tabular} &  & %
\begin{tabular}{c}
$1234|5|6|7$\tabularnewline
$16|25|34|7$\tabularnewline
\end{tabular} &  & %
\begin{tabular}{c}
$12|34|5|6|7$\tabularnewline
$1|25|34|6|7$\tabularnewline
\end{tabular}\tabularnewline
$(\mathcal{A}_{2},\mathcal{B}_{2})$ &  & $(\mathcal{A}_{5},\mathcal{B}_{5})$ &  & $(\mathcal{A}_{8},\mathcal{B}_{8})$\tabularnewline
 &  &  &  & \tabularnewline
\begin{tabular}{c}
$1234|567$\tabularnewline
$123456|7$\tabularnewline
\end{tabular} &  & %
\begin{tabular}{c}
$123|4|567$\tabularnewline
$124|356|7$\tabularnewline
\end{tabular} &  & %
\begin{tabular}{c}
$1|2|3|4|5|6|7$\tabularnewline
$1|2|34|5|6|7$\tabularnewline
\end{tabular}\tabularnewline
$(\mathcal{A}_{3},\mathcal{B}_{3})$ &  & $(\mathcal{A}_{6},\mathcal{B}_{6})$ &  & $(\mathcal{A}_{9},\mathcal{B}_{9})$\tabularnewline
\end{tabular}~ & %
\begin{tabular}{c}
\includegraphics{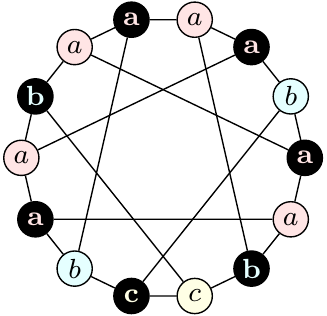}\tabularnewline
\end{tabular}\tabularnewline
(i) &  & (ii) & (iii)\tabularnewline
\end{tabular}

\caption{\label{fig:Fano}(i) The tactical decomposition $(\mathcal{A}_{4},\mathcal{B}_{4})$.
(ii) Orbit representatives of the tactical decompositions of $M$.
(iii) The tactical decomposition $(\mathcal{A}_{4},\mathcal{B}_{4})$
on the incidence graph.}
\end{figure}

The Fano plane is shown in Figure~\ref{fig:FanoDef}. The split and
cir algorithm finds 100 tactical decompositions of the incidence matrix.
 Figure~\ref{fig:Fano}(ii) shows the orbit representatives. The
self duality $\left(\begin{smallmatrix}1 & 2 & 3 & 4 & 5 & 6 & 7\\
2 & 5 & 3 & 4 & 1 & 6 & 7
\end{smallmatrix}\right)$ is visible as a horizontal reflection of the incidence graph. The
two line notation means that each point in the first row is swapped
with the corresponding line in the second row. This duality maps the
tactical decomposition $(\mathcal{A}_{2},\mathcal{B}_{2})$ to $(\mathcal{A}_{3},\mathcal{B}_{3})$,
maps $(\mathcal{A}_{5},\mathcal{B}_{5})$ to $(\mathcal{A}_{7},\mathcal{B}_{7})$,
and maps $(\mathcal{A}_{i},\mathcal{B}_{i})$ to itself for $i\in\left\{ 1,4,8,9\right\} $.
Finally, $(\mathcal{A}_{6},\mathcal{B}_{6})$ maps to itself by $\left(\begin{smallmatrix}1 & 2 & 3 & 4 & 5 & 6 & 7\\
4 & 2 & 1 & 7 & 5 & 3 & 6
\end{smallmatrix}\right)$. This self duality also acts as a reflection on the incidence graph.

For the tactical decomposition $(\mathcal{A}_{4},\mathcal{B}_{4})$,
shown in Figure~\ref{fig:Fano}(i), the first duality maps points
labeled $a$ to solid lines, points labeled $b$ to dashed lines,
and the point labeled $c$ to the dotted line. This corresponds to
the equitable partition of the incidence graph shown in Figure~\ref{fig:Fano}(iii).
\end{example}

\bibliographystyle{plain}
\bibliography{nss}

\end{document}